\documentclass[12pt]{article}
\usepackage{graphicx}
\usepackage{amsmath} \usepackage{amssymb} \usepackage{theorem}

\usepackage{enumitem}

\usepackage[nomessages]{fp} 

\usepackage{mathrsfs}
\usepackage[colorlinks=true,citecolor=blue,linkcolor=red,backref=page]{hyperref}
\usepackage{etoolbox} \makeatletter
\patchcmd{\BR@backref}{\newblock}{\newblock[cited on p.~}{}{}
\patchcmd{\BR@backref}{\par}{]\par}{}{} \makeatother

\numberwithin{equation}{section}

\newtheorem{thm}{Theorem}[section]\newtheorem{lemma}[thm]{Lemma}
\newtheorem{prop}[thm]{Proposition} \newtheorem{cor}[thm]{Corollary}
{\theorembodyfont{\rmfamily} 
   
  \newtheorem{rmk}[thm]{Remark} }

\renewcommand{\thesubsection}{\arabic{section}.\arabic{subsection}}

\newcommand{\qed}{\hfill \mbox{\raggedright \rule{.07in}{.1in}}}
 
\newenvironment{proof}{\vspace{1ex}\noindent{\bf
    Proof}\hspace{0.5em}}{\hfill\qed\vspace{1ex}}
\newenvironment{pfof}[1]{\vspace{1ex}\noindent{\bf Proof of
    #1}\hspace{0.5em}}{\hfill\qed\vspace{1ex}}

\newcommand{\R}{{\mathbb R}} 
\newcommand{\C}{{\mathbb C}}
\newcommand{\Z}{{\mathbb Z}} 
\newcommand{\T}{{\mathbb T}} 
\newcommand{\E}{{\mathbb E}}
\newcommand{\PP}{{\mathbb P}}
 \newcommand{\D}{{\mathbb D}}
\newcommand{\barD}{{\overline{\D}}}

\newcommand{\cB}{{\mathcal B}}

\newcommand{\cM}{{\mathcal M}}
 
\newcommand{\cH}{{\mathcal H}}

\newcommand{\hF}{{\widehat F}} 

\newcommand{\eps}{{\epsilon}}

\newcommand{\spec}{\operatorname{spec}}
\newcommand{\diam}{\operatorname{diam}}
\newcommand{\sgn}{\operatorname{sgn}}
\newcommand{\divv}{\operatorname{div}}
\newcommand{\BV}{{\operatorname{BV}}}
\newcommand{\Var}{\operatorname{Var}}
\newcommand{\Int}{\operatorname{Int}}
\newcommand{\Leb}{\operatorname{Leb}}

\newcommand{\SMALL}{\textstyle} \newcommand{\BIG}{\displaystyle}

\newcommand{\uex}{4} 
\FPeval{\uexminus}{clip(\uex-1)}
\FPeval{\uexsquare}{clip(\uex*uex)}
\FPeval{\uexsquareminus}{clip(\uex*uex-1)}
\FPeval{\twiceuex}{clip(2*\uex)}
\newcommand{\uexfrac}{\frac{\uexminus}{\uex}}
\newcommand{\uexfracone}{\frac{1}{\uex}}
\newcommand{\uexfractwo}{\frac{\uex}{\uexminus}}
\newcommand{\uexfracsquare}{\frac{\uexsquareminus}{\uexsquare}}

\newcommand{\graph}{\operatorname{graph}}

\newcommand\ve{\varepsilon}
\newcommand{\epstail}{\ve_0}

 \newcommand{\curvlen}{1+\delta}

\newcommand\metr{\mathbf{d}} 

\title{Sharp Statistical Properties for a Family of Multidimensional NonMarkovian Nonconformal Intermittent Maps.}

\author{ Peyman Eslami
\thanks{Dipartimento di Matematica, II Universit\`a di Roma (Tor Vergata),
00133 Roma, Italy.
\newline  peslami7@gmail.com)
}
\and Ian Melbourne
\thanks{Mathematics Institute, University of Warwick, Coventry, CV4 7AL, UK.
\newline i.melbourne@warwick.ac.uk}
 \and Sandro Vaienti 
\thanks{Aix Marseille Universit\'e, Universit\'e de Toulon, CNRS, CPT, 13009 Marseille, France. vaienti@cpt.univ-mrs.fr}
}

\date{24 March 2019.  Updated 30 June 2021.}

\begin{document}

\maketitle

\begin{abstract}
  Intermittent maps of Pomeau-Manneville type are well-studied in one-dimension,
and also in higher dimensions if the map happens to be Markov.
In general, the nonconformality of multidimensional intermittent maps represents a challenge that up to now is only partially addressed.
We show how to prove sharp polynomial bounds on decay of correlations 
for a class of multidimensional intermittent maps.
In addition we show that the optimal results on  statistical limit laws for one-dimensional intermittent maps hold also for the maps considered here. This includes the (functional) central limit theorem and local limit theorem, Berry-Esseen estimates, large deviation estimates, convergence to
 stable laws and L\'evy processes, and infinite measure mixing.
\end{abstract}

\section{Introduction} 
\label{sec:intro}

Intermittent maps were introduced by Pomeau \& Manneville~\cite{PomeauManneville80} as a model for turbulence.  These are maps that are uniformly expanding except for the presence of neutral fixed points.
In the smooth ergodic theory literature, they have provided the archetypal examples of nonuniformly expanding dynamical systems.  For one-dimensional intermittent maps,~\cite{Thaler80}
studied the invariant densities in the case when the map is Markov with respect to a suitable partition, and the nonMarkovian case was analysed in~\cite{Zweimuller98}.

The paper of Liverani, Saussol \& Vaienti~\cite{LSV99} set out to study the statistical properties of one-dimensional intermittent maps by considering the simplest possible example
$f:[0,1]\to[0,1]$, namely
\begin{align} \label{eq:LSV}
  f(x)=\begin{cases} x(1+2^\gamma x^\gamma), & 0\le x\le \frac12 \\
    2 x-1, & \frac12 < x\le1\end{cases}.
\end{align}
Here $\gamma>0$ is a real parameter.  For $\gamma\in(0,1)$, there
is a unique absolutely continuous probability measure $\mu$.
Let 
\begin{align} \label{eq:rho}
\rho_{v,w}(n)=\int v\,w\circ f^n\,d\mu-\int v\,d\mu\int w\,d\mu.
\end{align}
By~\cite{Hu04,Young99}, $\rho_{v,w}(n)=O(n^{-(\frac 1 \gamma -1)})$
for $v$ H\"older and $w\in L^\infty$ and this decay rate is
optimal~\cite{Gouezel04a,Sarig02}.
For $\gamma\in(0,\frac12)$, the central limit theorem (CLT)
holds for H\"older observables by~\cite{LSV99,Young99} as does the weak invariance principle (WIP)~\cite{MN05}. Berry-Esseen estimates and local limit theorems were obtained in~\cite{Gouezel05}.
When $\gamma\in[\frac12,1)$, the CLT fails for H\"older observables that are nonzero at $x=0$; stable laws were proved in this situation by~\cite{Gouezel04}
and the corresponding WIP holds by~\cite{MZ15}.

In addition, for $\gamma\in(0,1)$, sharp results on large deviations and convergence of moments were obtained in~\cite{DedeckerMerlevede15,GouezelM14,M09,MN08,MTorok12}.

For $\gamma\ge1$, there is a unique absolutely continuous invariant $\sigma$-finite measure up to scaling, but the measure is infinite.  Results on mixing for infinite measure systems were obtained in~\cite{Gouezel11,MT12}

Although~\cite{LSV99} initially focused on the specific maps~\eqref{eq:LSV},
the results described above have by now been shown to hold for very general classes of one-dimensional intermittent maps and extend to many multi-dimensional examples in cases when the map $f$ is Markov.  
In such cases, the standard approach is to construct an induced map $F$ with infinitely many branches and to deduce quasicompactness properties of the transfer operator for $F$ acting on a suitable function space.
In the Markov case, it is natural to consider observables that are H\"older with respect to a symbolic metric; in the one-dimensional case, one can consider observables of bounded variation.

Currently, multidimensional intermittent maps are poorly understood in general. 
The aim of this paper is to approach the problem of multidimensional intermittent maps in the same spirit that~\cite{LSV99} approached one-dimensional intermittent maps, focusing on some simple examples that exhibit all the problematic features: multidimensional, intermittent, nonconformal, nonMarkovian.  

\subsection{Statement of the main results}
Let $M=[0,1]\times\T$ where $\T=\R/\Z$.
Our counterpart of the family~\eqref{eq:LSV}
is the family of maps $f:M\to M$ with
$f(x,\theta)=(f_1(x,\theta),f_2(\theta))$, where
$f_1:M\to[0,1]$ is a (not
necessarily Markov) nonuniformly expanding map for each $\theta$.
Specifically, we assume that 
\begin{align} \label{eq:EMV}
  f_1(x,\theta)=\begin{cases} x(1+x^\gamma u(x,\theta)), & 0\le x\le \uexfrac \\
    \uex x-\uexminus, & \uexfrac < x\le1\end{cases}, \qquad
f_2(\theta)=\uex\theta\bmod1,
\end{align} 
where $\gamma>0$, and $u:[0,\uexfrac]\times\T\to(0,\infty)$ is a
positive $C^2$ function satisfying $u(0,\theta)\equiv c_0>0$.
(Implicitly, it is assumed that 
$x(1+x^\gamma u(x,\theta))\le1$ for all $x\in[0,\uexfrac]$, $\theta\in\T$.)
In addition, we assume that $|(Df)_{(x,\theta)}v|\ge|v|$
for all $(x,\theta)\in[0,\uexfrac]\times\T$, $v\in\R^2$.

In particular, as in~\cite{LSV99}, 
$f$ is an everywhere expanding map with a 
neutral invariant circle $\{x=0\}$, and $f$ is uniformly expanding on
$[\delta,1]\times\T$ for all $\delta>0$.  Also,
$f_1(\uexfrac,\theta)>\uexfrac$ for $\theta\in\T$.
For definiteness, we suppose that 
$f_1(\uexfrac,\theta)>\uexfracsquare$ for $\theta\in\T$.
Our final assumption is that $u$ is sufficiently close to constant, in
  the sense that $|x\frac{\partial u}{\partial x}|_\infty$ and
  $|\frac{\partial u}{\partial \theta}|_\infty$ are sufficiently
  small.  
(See Remarks~\ref{rmk:u} and~\ref{rmk:boundary}.)

Note that $f$ has $8$ branches; again as in~\cite{LSV99} half of the branches are linear.   However, typically the remaining branches are not full and there is no Markovian structure.   Moreover, the maps expand polynomially in $x$ and exponentially in $\theta$ and hence are highly nonconformal.

For the maps~\eqref{eq:EMV}, we obtain almost identical results to the ones described above for the one-dimensional maps~\eqref{eq:LSV}.  
Recall that $v:M\to\R$ is H\"older with exponent $\eta\in(0,1)$, denoted $v\in C^\eta(M)$, if
$\|v\|_\eta=|v|_\infty+\sup_{x\neq y}|v(x)-v(y)|/|x-y|^\eta$ is finite.
Our results are formulated mainly for H\"older observables, but occasionally for observables in $\BV_\infty(M)=\BV(M)\cap L^\infty(M)$.  (The definition of bounded variation on $M$ is recalled in Section~\ref{sec:BV}.)
In particular, all of the results hold for $C^1$ observables.

Lemma~\ref{lem:fmix} states that for $\gamma<1$, there is a unique absolutely continuous $f$-invariant probability measure, denoted $\mu$, and this measure is mixing.
Our main result gives sharp polynomial upper and lower bounds on the rate of mixing.
We set $\alpha=1/\gamma$ throughout. 
Define the correlation function $\rho_{v,w}$ as in~\eqref{eq:rho}.

\begin{thm} \label{thm:decay}
Suppose that $\gamma<1$.
\begin{itemize}
\item[(a)] 
 Let $\eta\in(0,1)$.  There exists $C>0$ such that
\[
|\rho_{v,w}(n)|\le C{\|v\|}_{\eta}|w|_\infty\,n^{-(\alpha-1)}
\quad\text{for all $n\ge1$},
\]
for all $v\in C^\eta(M)$, $w\in L^\infty(M)$.
\item[(b)]  
Define
\(
E(n)=\begin{cases} n^{-\alpha} & \alpha>2
\\ n^{-2}\log n & \alpha=2 \\ n^{-2(\alpha-1)}
& 1<\alpha<2
\end{cases}.
\)
There exists $C>0$, $c>0$ such that
\[
\Big|\rho_{v,w}(n)-cn^{-(\alpha-1)}\int v\,d\mu\int w\,d\mu\Big|\le CE(n)({\|v\|}_\BV+|v|_\infty)|w|_1
\quad\text{for all $n\ge1$},
\]
for all $v\in \BV_\infty(M)$, $w\in L^1(M)$ supported in $[\uexfrac,1]\times\T$.
In particular,
\(
\rho_{v,w}(n)\sim cn^{-(\alpha-1)}\int v\,d\mu \int w\,d\mu
\)
as $n\to\infty$.
\item[(c)]  There exists $C>0$ such that
\[
|\rho_{v,w}(n)|\le C({\|v\|}_\BV+|v|_\infty)|w|_1 n^{-\alpha}
\quad\text{for all $n\ge1$},
\]
for all $v\in \BV_\infty(M)$, $w\in L^1(M)$ 
supported in $[\uexfrac,1]\times\T$ with $\int v\,d\mu=0$.
\end{itemize}
\end{thm}

For $\gamma<\frac12$, corresponding to summable decay of correlations in Theorem~\ref{thm:decay}(a), we obtain the CLT and related results.
Define $v_n=\sum_{j=0}^{n-1}v\circ f^j$.
Also, define $W_n(t)=n^{-1/2}v_{nt}$ for $t=0,\frac1n,\frac2n,\dots,1$ and linearly interpolate to obtain $W_n\in C[0,1]$.

\begin{thm} \label{thm:stats}
Suppose that $\gamma<\frac12$.  Let $v:M\to\R$ be H\"older with $\int v\,d\mu=0$.
\begin{itemize}
\item[(a)] {\bf CLT} $n^{-1/2}v_n$ converges in distribution\footnote{Here and elsewhere, convergence in distribution (or weak convergence) holds on the probability space $(M,\mu)$ and equivalently~\cite{Zweimuller07} on the probability space $(M,\Leb_M)$
where $\Leb_M$ denotes normalised Lebesgue measure on the support of $\mu$.}
 to a normal distribution $G=_d N(0,\sigma^2)$.  The variance $\sigma^2$ is zero if and only if $v=\chi\circ f-\chi$ for some $\chi$ measurable.
\item[(b)] {\bf Berry-Esseen} There exists $C>0$ such that 
\[
|\mu(n^{-1/2}v_n\le a)-\PP(G\le a)|\le Cn^{-q}
\quad\text{for all $a\in\R$, $n\ge1$,}
\]
where $q=\frac12$ for $\alpha>3$ and $q=(\alpha-2)/2$ for $\alpha\in(2,3)$
(Any $q<\frac12$ works for $\alpha=3$.)
\item[(c)]  {\bf Local limit theorem} Suppose that $v$ is aperiodic\footnote{Aperiodic means that it is not possible to write $v\equiv\chi-\chi\circ T+{\rm constant}\bmod \lambda\Z$ for some $\chi$ measurable and $\lambda>0$.}.
For all $a,b,\kappa\in\R$ with $a<b$, all $k_n\in\R$ with $k_n\sim \kappa n^{1/2}$, and all $u\in C^\eta(M)$, $w:M\to\R$ measurable,
\[
\lim_{n\to\infty}n^{1/2}\mu\big\{x\in M:v_n(x)-\kappa_n-u(x)-w(f^nx)\}\in[a,b]\big\}=(b-a)\dfrac{e^{-\kappa^2/(2\sigma^2)}}{(2\pi\sigma^2)^{1/2}}.
\]
\item [(d)] {\bf WIP}
$W_n$ converges weakly in $C[0,1]$ to Brownian motion $W$ with $W(1)=_d G$.
\item[(e)] {\bf Error rate in WIP} For any $q<(\alpha-2/(4\alpha)$, there exists $C>0$ such that
$\pi_1(W_n,W)\le Cn^{-q}$ for all $n\ge1$.
\footnote{Let $A^\eps$ denote the $\eps$-neighborhood of $A$. The Prokhorov metric $\pi_1$ is given by \newline
$\pi_1 (X, Y ) = \inf\{\eps > 0 : \PP(X \in A) \le  \PP(Y \in A^\eps ) +\eps
\text{ for all closed sets $A \subset C[0,1]$}
\}$.}
\item[(f)] {\bf Almost sure invariance principle}
For any $\eps>0$, there is a probability space supporting $W$ and a sequence of random variables $\{\tilde v_n;\,n\ge1\}$ 
with the same joint distributions as $\{v_n\}$ such that
$\tilde v_n=W(n)+O(n^\gamma(\log n)^{\gamma+\eps})$ a.e.
\end{itemize}
\end{thm}

For $\gamma\in(\frac12,1)$, the CLT with normalization $n^{-1/2}$ fails for general H\"older observables,  and we obtain results on anomalous diffusion.
Let $G_\alpha$ denote the totally skewed $\alpha$-stable law with
characteristic function
$\E(e^{itG_\alpha})= \exp\{-|t|^\alpha(1-i\sgn t\tan{\SMALL\frac{\alpha\pi}{2}})\}$.

\begin{thm} \label{thm:stab}
Let $v:M\to\R$ be H\"older with $\int v\,d\mu=0$.
Suppose that 
\mbox{$\int_\T v(0,\theta)\,d\theta\neq0$}.
Then there exists $c>0$ such that
$n^{-1/\alpha}v_n$ converges in distribution to $cG_\alpha$.

Moreover the process defined by $W_n(t)=n^{-1/\alpha}v_{[nt]}$ converges weakly in $D[0,1]$ with the $\cM_1$ topology\footnote{We refer to~\cite{Skorohod56,Whitt} for background information on $D[0,1]$ and the Skorohod $\cM_1$ topology.} to the $\alpha$-stable L\'evy process $W$ with $W(1)=_d cG_\alpha$.
\end{thm}

Next, we consider large deviation estimates and moment estimates.
\begin{thm} \label{thm:LD}
Suppose that $\gamma<1$ and let $v:M\to\R$ be H\"older.
\begin{itemize}
\item[(a)] {\bf Large deviation estimates}
For any $a>0$, there exists $C>0$ such that
\[
\mu\Big\{\Big|\frac1n v_n-\int v\,d\mu\Big|>a\Big\}\le Cn^{-(\alpha-1)}
\quad\text{for all $n\ge1$}.
\]
\item[(b)] {\bf Moment estimates}  For any $p\ge1$, there exists $C>0$
such that for all $n\ge1$ 
\[
\int|v_n|^p\,d\mu
\le C\max\{g(n),n^{p-\alpha+1}\}
\quad\text{where}\quad
g(n)=\begin{cases} n^{p/2} & \alpha>2 \\ (n\log n)^{p/2} & \alpha=2 
\\ n^{p/\alpha} & 1<\alpha<2,\,p\neq\alpha
\\ n\log n & 1<\alpha<2,\,p=\alpha
 \end{cases}.
\]
\item[(c)] {\bf Convergence of moments}  
If $\gamma<\frac12$, then
$\int |n^{-1/2}v_n|^p\,d\mu\to \E|G|^p$ for all $p<2(\alpha-1)$.

If $\gamma\in(\frac12,1)$, then
$\int |n^{-1/\alpha}v_n|^p\,d\mu\to \E|cG_\alpha|^p$ for all $p<\alpha$
where $c$ is the constant in Theorem~\ref{thm:stab}.
\end{itemize}
\end{thm}

For $\gamma\ge1$, Lemma~\ref{lem:fmix} states that up to scaling there is a unique absolutely continuous $f$-invariant $\sigma$-finite measure $\mu$, but now $\mu(M)=\infty$.  We prove the following mixing property for infinite measure systems.

\begin{thm} \label{thm:infinite}
\begin{itemize}
\item[(a)]
Suppose that $\gamma>1$.  There exists $c>0$ such that
\[
\lim_{n\to\infty}n^{1-\alpha}\int v\,w\circ f^n\,d\mu = c \int v\,d\mu \int w\,d\mu,
\]
for all $v\in\BV_\infty(M)$, $w\in L^1(M)$ 
supported in $[\uexfrac,1]\times\T$.

For $\gamma=1$, the same result holds with 
$n^{1-\alpha}$ replaced by $\log n$.
\item[(b)]
For $\gamma>1$, there exists $C>0$ such that
\[
\Big|\int v\,w\circ f^n\,d\mu\Big|\le C({\|v\|}_\BV+|v|_\infty)|w|_1 n^{-\alpha}
\quad\text{for all $n\ge1$},
\]
for all $v\in\BV_\infty(M)$, $w\in L^1(M)$
supported in $[\uexfrac,1]\times\T$ with $\int v\,d\mu=0$.
\end{itemize}
\end{thm}

\begin{rmk} The constants $c$ in Theorems~\ref{thm:decay}(b),~\ref{thm:stab} and~\ref{thm:infinite}(a) are given explicitly in Sections~\ref{sec:R} and~\ref{sec:levy}.
\end{rmk}

\begin{rmk}
It is an easy but tedious exercise to extend to cases where $\theta$ is of general dimension
and $f_2:\T^{d-1}\to\T^{d-1}$ is a general smooth uniformly expanding map with worst expansion sufficiently large  (strictly larger than $3$ suffices when $d=2$),
but we restrict to the current situation for readability.

A notationally simpler example would have $f_2(\theta)=2\theta\bmod1$, but it is well-known that the extra expansion is useful  in higher dimensions.
Our assumption that $u$ is sufficiently close to constant is of the same flavour and can be relaxed by assuming sufficient expansivity of $f$, see Remark~\ref{rmk:boundary}.
\end{rmk}

\subsection{Comparison with other results and methods}

There is a considerable amount of work on uniformly expanding maps in higher dimensions.  In the analytic setting, see~\cite{Buzzi00,Tsujii00}.
For $C^2$ maps, still with finitely many branches, see~\cite{Cowieson00,Cowieson02,GoraBoyarsky89}.
The paper~\cite{Liverani13} sets
out a general approach to multidimensional uniformly expanding maps with infinitely many branches.  This method could in principle be applied to the first return maps $F$ mentioned in Subsection~\ref{sec:F}.  However, the assumptions therein do not hold for the examples in~\cite{HuVaienti09,HuVaientiapp} nor the examples~\eqref{eq:EMV}.
(Condition~4 in~\cite{Liverani13} fails due to the lack of conformality; the condition has the form $\lim_{\eps\to0}A_\eps=0$ but in our examples $A_\eps=\infty$ for $\eps>0$.)  Another approach in~\cite{Saussol00} uses quasi-H\"older spaces, but these also have drawbacks as discussed below.

Turning to multidimensional intermittent maps, we mention the work of~\cite{BahsounBoseDuan14,BahsounBose16,BahsounBoseRuziboev19,Gouezel07} which treats examples like those in~\eqref{eq:EMV} with $\gamma$ depending on $\theta$.  However, these papers require that $f$ is Markovian and hence do not encounter the issues treated here.

In contrast, there are has been very little work on multidimensional nonMarkovian nonuniformly expanding maps.  We now list all papers on this topic that we know of.
A large class of multidimensional intermittent maps was considered in~\cite{HuVaienti09,HuVaientiapp} using the quasi-H\"older spaces from~\cite{Saussol00}.  In particular,~\cite{HuVaienti09} obtained results on existence of absolutely continuous invariant measures, but it was convenient to consider maps that were close to conformal.   For statistical limit laws, it seems that quasi-H\"older spaces handle nonconformality of multidimensional maps quite poorly.  A more recent paper~\cite{BMTsub} obtains almost optimal, but still nonoptimal, results on decay of correlations for the maps
in~\cite{HuVaienti09,HuVaientiapp}.  Moreover, the methods in~\cite{BMTsub} do
not seem to apply to the maps~\eqref{eq:EMV} considered here.

\subsection{Structure of the paper}
\label{sec:F}
The method in this paper starts off, as usual, by constructing a convenient
first return map $F:Y\to Y$, and from then on is a hybrid of two standard methods.  
Reinducing enables us to model $f$ by a Young tower with polynomial tails, leading to existence of absolutely continuous invariant measures and a spectral decomposition.  For $\gamma\in(0,\frac12)$, this already yields sharp upper bounds on decay of correlations as well as a number of statistical limit laws.  Combining the information on invariant measures with
bounded variation methods for $F$, we obtain sharp lower bounds on decay of correlations, 
as well as 
convergence to stable laws and L\'evy processes, and results on infinite measure mixing.  
This hybrid method bypasses many of the problems associated with multi-dimensional bounded variation (namely, that the function space is not contained in $L^\infty$; supports of invariant densities need not {\em a priori} have nonempty interior; certain aperiodicity assumptions are hard to verify).

The reinducing step, Lemma~\ref{lem:M} below, makes use of recent
work~\cite{Esl19} based on the method of standard pairs \cite{Che1, Dol00}, and gives precise joint control on the first return time to $Y$ and the reinducing return time (denoted respectively as $\varphi$ and $\rho$ below).   As already noted, the reinducing approach adopted in~\cite{BMTsub} seems not applicable for the examples in this paper and in any case gives much less control on return times.

The remainder of this paper is organised as follows.
In Section~\ref{sec:calc}, we 
construct a convenient first return map~$F$
and obtain estimates for the first return time and distortion bounds for $F$.  In Section~\ref{sec:mix},
we derive mixing properties of \(f\) and \(F\) and results on aperiodicity.
In the process of doing this, we show that $f$ can be modelled by a Young tower with polynomial decay of correlations.  We use this to prove Theorems~\ref{thm:decay}(a),~\ref{thm:stats} and~\ref{thm:LD}.  

Section~\ref{sec:BV}
contains functional analytic estimates in bounded variation.
In Section~\ref{sec:R}, we prove Theorems~\ref{thm:decay}(b,c) and~\ref{thm:infinite}.  Finally, we prove Theorem~\ref{thm:stab} 
in Section~\ref{sec:levy}.

\paragraph{Notation}
We use the ``big $O$'' and $\ll$ notation interchangeably, writing $a_n=O(b_n)$ or $a_n\ll b_n$ if there is a constant $C>0$ such that
$a_n\le Cb_n$ for all $n\ge1$.  
Also, $a_n=o(b_n)$ as $n\to\infty$ means that $\lim_{n\to\infty}a_n/b_n=0$ and
$a_n\sim b_n$ as $n\to\infty$ means that $\lim_{n\to\infty}a_n/b_n=1$.

We set $\barD=\{\omega\in\C:|\omega|\le1\}$.
Throughout, $|\cdot|$ denotes Euclidean distance.

\section{Estimates for the first return map $F$}
\label{sec:calc}

\subsection{Construction of $F$}
Let $f:M\to M$, $M=[0,1]\times\T$, belong to the class of maps~\eqref{eq:EMV}.

Define
\[
  \SMALL X_i=\{(x,\theta)\in M : 0\le x\le
  f_1(\uexfrac,\frac{i+\theta}{\uex})\}
  =f([0,\uexfrac]\times[\frac{i}{\uex},\frac{i+1}{\uex}]), \quad i=0,1,2, \uexminus.
\]
Then $X=\bigcup_{i=0}^{\uexminus}X_i$ is an invariant set for $f$ and $f(X)=X$.

We induce on the set $Y=([\uexfrac,1]\times\T)\cap X$.  Let
$\varphi:Y\to\Z^+$ be the first return time, with first return map
$F=f^\varphi:Y\to Y$.  The sets
\[
  Y_{n,j}=\{(y,\theta)\in Y:\varphi(y,\theta)=n:(j-1)/\uex^n< \theta
  < j/\uex^n\}, \quad n\ge1,\,1\le j\le \uex^n,
\]
form a (mod $0$) partition $\alpha^Y$ of $Y$.  Note that $F:a\to Fa$
is a diffeomorphism for each $a\in\alpha^Y$.
We have 
\begin{equation} \label{eq:Y1}
\SMALL Y_{1,j}=\{(y,\theta)\in Y: \uexfracsquare< y< f_1(\uexfrac),\,(j-1)/\uex< \theta < j/\uex\}.  
\end{equation}
Also,
$FY_{n,j} \in \{([\uexfrac,1]\times\T)\cap X_{i}\}_{i=0}^{\uexminus}$
 for $n \ge 2$.
 In particular, $F$ has finitely many images, i.e.\ $\{Fa:a\in\alpha^Y\}$ is finite.

Figure~\ref{fig-F} is a sketch of $f(\cdot,\theta)$ and
$F(\cdot,\theta)$ for $\theta$ fixed.  Figure~\ref{fig-a} is a
schematic picture of the partition
$\alpha^Y=\{Y_{n,j}:n\ge1,\,1\le j\le \uex^n\}$.
\begin{figure}
  \centering
 \includegraphics[width = 0.8\columnwidth, keepaspectratio]{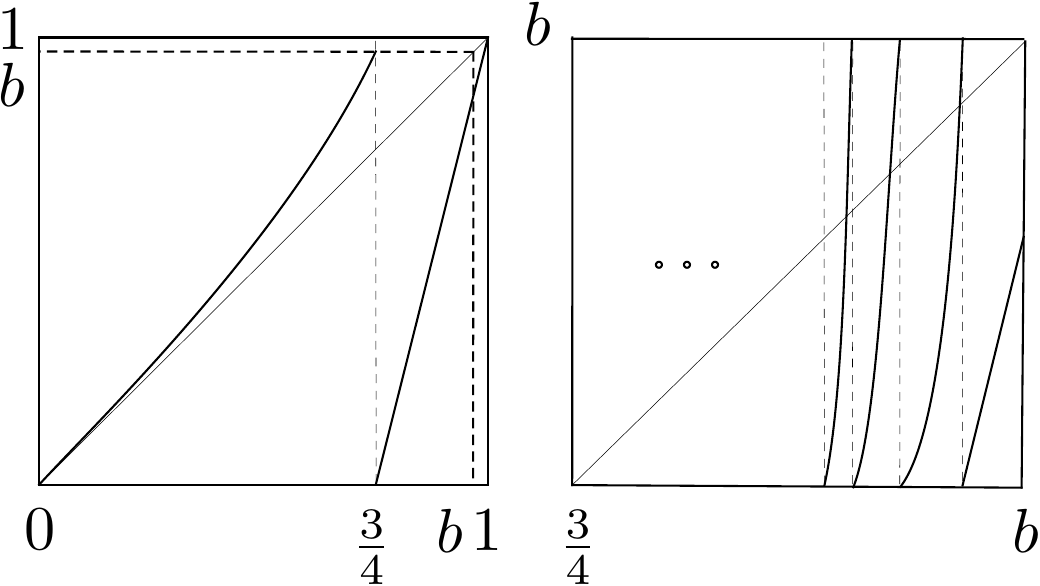}
  \caption{The maps $f$ and $F$ for fixed $\theta$.  The letter $b$
    denotes $f_1(\uexfrac,\theta)$}
  \label{fig-F}
\end{figure}

\begin{figure}
  \centering
  \includegraphics[width = 0.4\columnwidth, keepaspectratio]{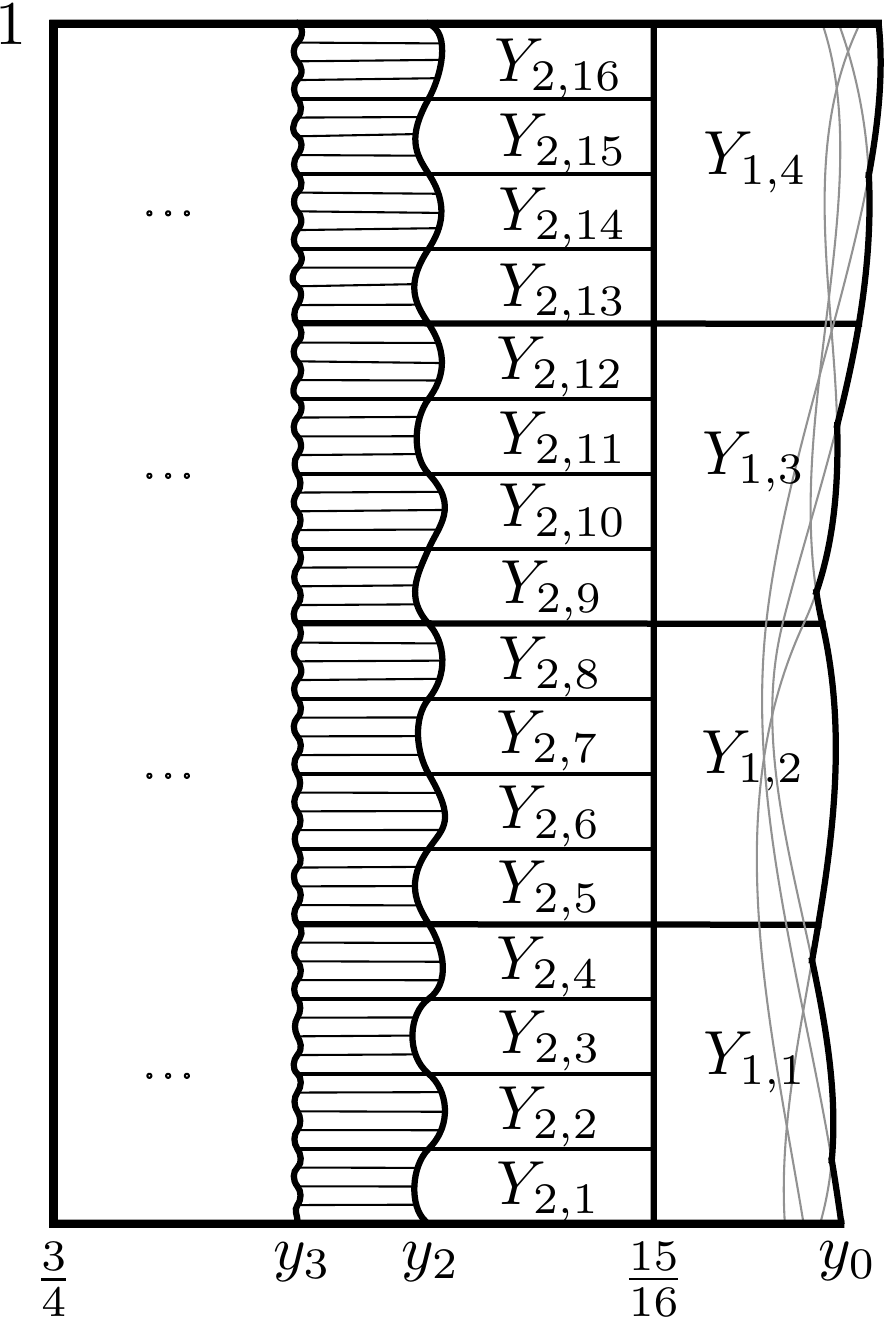}
  \caption{The partition $\alpha^Y=\{Y_{n,j},\,n\ge1,\,1\le j\le \uex^n\}$}
  \label{fig-a}
\end{figure}

\begin{prop} \label{prop:expand}
$|(DF)_{(y,\theta)}v|\ge \uex|v|$ for all $(y,\theta)\in Y$, $v\in\R^2$.
\end{prop}

\begin{proof} 
We have $Df=\left(\begin{array}{cc} \uex & 0 \\ 0 & \uex \end{array}\right)$ on $Y$
and 
$|(Df)v|\ge |v|$ on $X$.
\end{proof}

\begin{prop} \label{prop:ftopmix} $f:X\to X$ is topologically exact:
  for any nonempty open subset $U\subset X$, there exists $n\ge0$ such
  that $f^nU\supset X\setminus\{x=0\}$.
\end{prop}

\begin{proof}
  It suffices to consider rectangles $U=U_1\times U_2$ where
  $U_1\subset [0,1]$, $U_2\subset \T$ are intervals.  Let
  $\pi:M\to[0,1]$ be projection onto the first coordinate.
  Note that
  \begin{equation} \label{eq:mix} \SMALL \text{For any
      $(x,\theta)\in(0,\uexfrac)\times\T$, there exists $n\ge1$ such
      that $\pi f^n(x,\theta)> \uexfrac$.}
  \end{equation}

  If $\uexfrac\in\pi U$, then $0\in\overline{\pi fU}$.  Since $0$ is a
  fixed point and $f_1$ is continuous on $[0,\uexfrac]\times\T$, it
  follows from~\eqref{eq:mix} that $(0,\uexfrac]\subset \pi f^nU$ for
  all $n$ sufficiently large.  Also $f_2:\T\to\T$ is continuous and
  uniformly expanding, so it is immediate that
  $(0,\uexfrac]\times\T\subset f^nU$ for some $n$ and hence that
  $f^{n+1}U\supset X\setminus\{x=0\}$.

  For a general rectangle $U$, it remains to show that
  $\uexfrac\in \pi f^nU$ for some $n\ge0$.  Suppose this is not the
  case.  Since $f_1$ is continuous on $[0,\uexfrac)\times\T$ and
  $(\uexfrac,1]\times\T$, it follows that $\pi f^nU$ is an interval for
  all $n$.  By~\eqref{eq:mix}, $\pi f^nU\subset(\uexfrac,1]$ infinitely
  often.  On each such occasion, $\diam \pi f^{n+1}U=\uex\diam \pi f^nU$,
  so $\diam \pi f^nU\to\infty$ which is impossible.
\end{proof}

\begin{prop} \label{prop:Ftopmix} $F:Y\to Y$ is topologically exact.
\end{prop}

\begin{proof}
  Let $U\subset Y$ be a nonempty open rectangle.  If $U$ intersects
  the boundary of the strip $\{\varphi=n\}$ for some $n$, then
  $\overline{FU}\cap(\{\uexfrac\}\times\T)\neq\emptyset$. But then $FU$
  contains a partition element $Y_{n,j}$. 
It follows that $F^2U\supset FY_{n,j}=Y\cap X_i$ for some $i$ and hence that 
$F^3U=Y$.

Again, let $\pi$ denote projection onto the first coordinate.
  If $F^kU$ does not intersect the boundary of $\{\varphi=n\}$ for all
  $n$ and all $k\ge0$, then $\diam \pi F^kU\ge \uex^k\diam \pi U$ for all $k$,
  which is impossible.
\end{proof}

\subsection{Estimates for the partition}
\label{sec:partition}

Recall that \(u:[0,\uexfrac]\times\T\to(0,\infty)\) is \(C^{2}\) and
$u(0,\theta)\equiv c_0>0$.  In fact, we only use the following consequences of this property:
\[
  u(x,\theta)-c_0=o(1), \quad x\frac{\partial u}{\partial
    x}(x,\theta)=o(1), \quad\text{and}\quad x^2\frac{\partial^2
    u}{\partial x^2}(x,\theta)=o(1),
\]
as $x\to0$ uniformly in $\theta$.

\begin{prop} \label{prop:xn} Suppose that $(x_n,\theta_n)$, $n\ge1$, is a
  sequence in $[0,\uexfrac]\times\T$ such that
  $f^n(x_n,\theta_n)=(\uexfrac,\theta)$ where $\theta\in\T$.  
Then
  \[
    x_n\sim c_1n^{-\alpha} \enspace\text{and}\enspace
    x_n-x_{n+1}\sim c'n^{-(1+\alpha)} \enspace\text{as
      $n\to\infty$},
  \]
  uniformly in $\theta$, where $c_1=(c_0\gamma)^{-\alpha}$ and
  $c'=c_1^{1+\gamma} c_0$.

  In addition, the curve $\theta\to x_n(\theta)$ is $C^1$ 
and 
  there exists a constant $C>0$ independent of
  $n,\theta$ such that $|x_n'(\theta)|\le Cn^{-(1+\alpha)}$.
\end{prop}

\begin{proof}   
  By construction, for each choice of inverse sequence $\theta_n$, the
  sequence $x_n$ is unique and monotonically decreasing to zero.  We
  do the computation for $\theta_n=\theta/\uex^n$, but the result is
  independent of this choice.  Write $\theta_0=\theta$.

  The inverse branch
  $\psi:M\to[0,\uexfrac]\times[0,\uexfracone]$ has the form
  $\psi(x,\theta)=(\psi_1(x,\theta),\uexfracone\theta)$.  Compute that
  \[
    \begin{split}
    \psi_1(x,\theta) & =
                       x(1-x^{\gamma} \tilde u(x,\theta))
                       =[x^{-{\gamma}}(1-x^{\gamma} \tilde u(x,\theta))^{-\gamma}]^{-\alpha} 
    \\ & =[x^{-{\gamma}}+\gamma\hat u(x,\theta)]^{-\alpha},
  \end{split}
  \]
  where $\hat u(x,\theta)=c_0+o(1)$ as $x\to0$ uniformly in $\theta$.
  Inductively,
  \[
    [\psi^n]_1(x,\theta) =
    \Bigl[x^{-{\gamma}}+\gamma\sum_{j=0}^{n-1}\hat u(\psi^j(x,\theta))\Bigr]^{-\alpha}.
  \]
  In particular,
  \begin{equation} \label{eq:induct}
    \begin{split}
      x_n &
    =[\psi^n]_1({\SMALL\uexfrac},\theta_0)=
    \Bigl[({\SMALL\uexfractwo})^{\gamma}+\gamma\sum_{j=0}^{n-1}\hat u(x_j,\theta_j)\Bigr]^{-\alpha}
=
    \Bigl[\gamma\sum_{j=0}^{n-1}\hat u(x_j,\theta_j)+O(1)\Bigr]^{-\alpha}.
  \end{split}
  \end{equation}
  Since $x_n\to0$, we have $\hat u(x_j,\theta_j)\to c_0$ and hence
  $\sum_{j=0}^{n-1}\hat u(\psi^j(x,\theta))=nc_0+o(n)$.  Substituting
  this into~\eqref{eq:induct} yields the desired expression for $x_n$.

  Since $(x_n,\theta_n)=f(x_{n+1},\theta_{n+1})$, we have
  $x_n=x_{n+1}(1+x_{n+1}^\gamma u(x_{n+1},\theta_{n+1}))$ and so
  \[
    x_n-x_{n+1}=x_{n+1}^{1+\gamma}u(x_{n+1},\theta_{n+1})\sim
    c_1^{1+\gamma}n^{-(1+\alpha)}c_0=c'n^{-(1+\alpha)}.
  \]

  Differentiating the formula for $x_n$ in~\eqref{eq:induct},
  \[
    x_n'(\theta)=-
    \Bigl[({\SMALL\uexfractwo})^{\gamma}+\gamma\sum_{j=0}^{n-1}\hat u(x_j,\theta_j)\Bigr]^{-(1+\alpha)}\Bigl(
    \sum_{j=0}^{n-1}\frac{\partial \hat u}{\partial
      x}(x_j,\theta_j)x_j'(\theta)+ \sum_{j=0}^{n-1}\frac{\partial
      \hat u}{\partial\theta}(x_j,\theta_j)\uex^{-j}\Bigr).
  \]
  It is easy to verify that $\hat u$ inherits the properties
  $\BIG x\frac{\partial \hat u}{\partial x}=o(1)$ and
  $\BIG \frac{\partial \hat u}{\partial \theta}=O(1)$ imposed on
  $u$.  Hence, there is a constant $K>0$ such that
  \[
    |x_n'(\theta)| \le
    Kn^{-(1+\alpha)}+ Kn^{-(1+\alpha)}
\sum_{j=0}^{n-1}j^{1/\gamma}|x_j'(\theta)|
    \le Kn^{-(1+\alpha)}+ Kn^{-1}
\sum_{j=0}^{n-1}|x_j'(\theta)|.
  \]
The discrete version of Gronwall's inequality states that if $|b_n|\le C_1+C_2\sum_{j=0}^{n-1}|b_j|$, then $|b_n|\le C_1(1+C_2)^n$.
Hence
$|x_n'(\theta)|\le Kn^{-(1+\alpha)}(1+Kn^{-1})^n\ll n^{-(1+\alpha)}$.

For each $n\ge1$, we have established 
smoothness of the curve $x_n(\theta)$ and the estimate $|x_n'(\theta)|\le Cn^{-(1+\alpha)}$
except at finitely many points related to the partition into inverse branches of $f^n$.
Since $x_n(\theta)=\{(x,\theta)\in[0,\uexfrac]\times\T^1:f^n(x,\theta)\in\{\uexfrac\}\times\T\}$ is defined intrinsically on the cylindrical domain $[0,\uexfrac]\times\T^1$, independent of any choice of partition, these smoothness properties are uniform in $\theta\in\T$.
\end{proof}

\begin{rmk} \label{rmk:square} It follows from
  Proposition~\ref{prop:xn} that if
  $\theta,\theta'\in[(j-1)\uex^{-n},j\uex^{-n}]$ for some $j=1,\dots,\uex^n$,
  then $x_n(\theta)-x_{n+1}(\theta')\sim c'n^{-(1+\alpha)}$ uniformly in $j$.
\end{rmk}

The partition elements $Y_{1,j}$, $1\le j\le\uex$, are as in~\eqref{eq:Y1}.
The remaining partition elements 
$Y_{n,j}$, 
$n\ge2$, $1\le j\le \uex^n$, are given by
\[
  Y_{n,j}=\{(y,\theta):y\in(y_n(\theta),y_{n-1}(\theta)),\,\theta\in((j-1)/\uex^n,j/\uex^n)\},
\]
where $y_n(\theta)=\uexfracone(x_{n-1}(f_2(\theta))+\uexminus)$. Note
that \(y_n'(\theta) = x_{n-1}'(\theta)\).
  
By
Remark~\ref{rmk:square}, $Y_{n,j}$ is almost rectangular for $n$
large, and
$y_n(\theta)-y_{n+1}(\theta)\sim \uexfracone c' n^{-(1+\alpha)}$ uniformly in $\theta$.

\begin{rmk} \label{rmk:u} By choosing $u$ sufficiently close to constant, in
  the sense that $|x\frac{\partial u}{\partial x}|_\infty$ and
  $|\frac{\partial u}{\partial \theta}|_\infty$ are sufficiently
  small, we can arrange that 
$|x_n'(\theta)|$ is uniformly small in $n$ and $\theta$.
In fact, we require that $|x_n'(\theta)|< 7/\sqrt{72}$ for
  all $n$ and~$\theta$.  This turns out to be convenient for technical
  reasons, see Remark~\ref{rmk:boundary}.
\end{rmk}

\begin{cor} \label{cor:Leb}
  $\Leb(\varphi>n)\sim \uexfracone c_1 n^{-\alpha}$ as $n\to\infty$.  In
  particular, $\varphi\in L^1$ if and only if $\gamma<1$.  In
  addition, $\Leb(\varphi=n)\sim \uexfracone c'(n^{-(1+\alpha)})$.
\end{cor}

\begin{proof}
Let $C_n=\{x_n(\theta)\}$, \(n
\ge 0\), be the smooth curve defined in Proposition~\ref{prop:xn} and let $X_{n+1}$ be the region in $[0,\uexfrac]\times\T$ to the left of $C_n$.  
Then
  $X_{n+1}$ consists of precisely those points
  in $[0,\uexfrac]\times\T$ that require at least $n+1$ iterates of $f$
  to enter $Y$ for the first time.  
  By Proposition~\ref{prop:xn}, 
  $\Leb(X_n)\sim c_1n^{-\alpha}$, 
  and $\Leb(X_{n-1}\setminus X_n)\sim c'n^{-(1+\alpha)}$.

  Now, $f$ maps $\{\varphi=n\}$ onto $X_{n-1}\setminus X_n$ for $n\ge2$, and
  the mapping is $\uex$ to $1$.  Each branch is linear and scales areas
  by a factor of $\uex^2$, so
  $\Leb(\{\varphi=n\})=\uexfracone\Leb(X_{n-1}\setminus X_n)$.  Hence
  $\Leb(\varphi=n)\sim \uexfracone c'n^{-(1+\alpha)}$ and
  $\Leb(\varphi>n)=\uexfracone\Leb(X_n)\sim \uexfracone c_1n^{-\alpha}$.
\end{proof}

\begin{rmk}
  Suppose further that $u(x,\theta)=c_0+O(x^{\gamma})$ uniformly in
  $\theta$.  This property is inherited by $\tilde u$ and $\hat u$
  in the above calculations and we obtain that
  $x_n=c_1n^{-\alpha}(1+O(n^{-1}\log n))$ uniformly in $\theta$.
\end{rmk}

\subsection{Distortion estimates}

\begin{lemma} \label{lem:DF} Let $(y,\theta)\in Y_{n,j}$.  Then 
  \[
    \prod_{k=\ell}^m\frac{\partial f_1}{\partial x}(f^k(y,\theta))\sim
    \Bigl(\frac{n-\ell}{n-m}\Bigr)^{1+\alpha}
\quad\text{as $n\to\infty$}.
  \]
This estimate holds uniformly in $(y,\theta)\in Y_{n,j}$, in $j$, and in
$1\le \ell\le m\le n-1$.
\end{lemma}

\begin{proof}
  Let $(y_k,\theta_k)=f^k(y,\theta)$, $k=0,\dots,n-1$ and recall from
  Proposition~\ref{prop:xn} that
  $y_{n-k}\sim (c_0\gamma k)^{-\alpha}$ as $k\to\infty$ uniformly in
  the initial choice of $\theta$.  Now,
  \begin{align} \label{eq:f1x}
    \frac{\partial f_1}{\partial x}(x,\theta)=1+(1+\gamma)x^\gamma
    u(x,\theta) +x^{1+\gamma}\frac{\partial u}{\partial
      x}(x,\theta) =1+c_0(1+\gamma)x^\gamma(1+o(1)),
  \end{align}
  as $x\to0$ uniformly in $\theta$, so
  \[
    \log \frac{\partial f_1}{\partial x}(f^{n-k}(y,\theta))\sim
    c_0(1+\gamma)y_{n-k}^\gamma \sim (1+\alpha)k^{-1}.
  \]
  It follows that
  \[
    \log\prod_{k=\ell}^m\frac{\partial f_1}{\partial x}(f^k(y,\theta))
    =\sum_{k=n-m}^{n-\ell}\log\frac{\partial f_1}{\partial
      x}(f^{n-k}(y,\theta))
    =(1+\alpha)\sum_{k=n-m}^{n-\ell}b_k(\theta),
  \]
  where $b_k(\theta)\sim k^{-1}$ as $k\to\infty$ uniformly in
  $\theta$.  Hence
  $\log\prod_{k=\ell}^m\frac{\partial f_1}{\partial
    x}(f^k(y,\theta))\sim (1+\alpha)(\log(n-\ell)-\log(n-m))$ and
  the result follows.
\end{proof}

\begin{cor} \label{cor:DF} Let $(y,\theta)\in Y_{n,j}$.  Then
  \[
    (DF)_{(y,\theta)}=\left(\begin{array}{cc} A(y,\theta) & B(y,\theta) \\
                              0 & \uex^n \end{array}\right),
                        \]
                        where $A(y,\theta)\sim \uex n^{1+\alpha}$ as
                        $n\to\infty$ 
and $B(y,\theta)=O(\uex^n)$.
These estimates hold uniformly in $(y,\theta)\in Y_{n,j}$ and in $j$.
                      \end{cor}

\begin{proof}
  Let $z_k=f^k(y,\theta)$, $k=0,\dots,n-1$ and write
  \[
    (DF)_{(y,\theta)}=(Df)_{z_{n-1}}\cdots (Df)_{z_1}(Df)_{z_0}.
  \]
  Clearly, $(Df)_{z_k}$ is upper triangular with diagonal entries
  $\frac{\partial f_1}{\partial x}(z_k)$ and $\uex$.  Hence
  $(DF)_{(y,\theta)}$ has the required form with
  \[
    A(y,\theta)=\prod_{k=0}^{n-1}\frac{\partial f_1}{\partial x}(z_k)
    =\uex\prod_{k=1}^{n-1}\frac{\partial f_1}{\partial x}(z_k)
  \]
  The required asymptotics for $A(y,\theta)$ follows from Lemma~\ref{lem:DF}.

  Next, 
  \[
    \frac{\partial[f^{k+1}]_1}{\partial \theta}(z_0) =
    \frac{\partial f_1}{\partial x}(z_k) \frac{[\partial f^k]_1}{\partial
      \theta}(z_0)+
    \frac{\partial f_1}{\partial \theta}(z_k) 4^k,
  \]
so by induction,
  \[
    \frac{\partial[f^k]_1}{\partial \theta}(y,\theta) =
    \sum_{\ell=0}^{k-1}\Bigl[\prod_{m=\ell+1}^{k-1}\frac{\partial
      f_1}{\partial x}(z_m)\Bigr] \frac{\partial f_1}{\partial
      \theta}(z_{\ell})\uex^{\ell}.
  \]
  By Lemma~\ref{lem:DF},
  $\prod_{m=\ell}^{k-1}\frac{\partial f_1}{\partial x}(z_m)\ll
  [(n-\ell)/(n-k)]^{1+\alpha}$ and so
  \[
    \begin{split}
      \frac{\partial[f^k]_1}{\partial \theta}(y,\theta) & \ll
      (n-k)^{-(1+\alpha)}\sum_{\ell=1}^k
      (n-\ell)^{1+\alpha}\uex^\ell \\ & \ll
      (n-k)^{-(1+\alpha)}\sum_{\ell=n-k}^n
      \ell^{1+\alpha}\uex^{n-\ell} \ll (n-k)^{-(1+\alpha)}\uex^n,
    \end{split}
  \]
  yielding the required estimate for $B(y,\theta)$.
\end{proof}

Let $JF=\det DF$.  By Corollary~\ref{cor:DF},
$JF\sim \uex^{n+1} n^{1+\alpha}$ on $Y_{n,j}$.

\begin{lemma} \label{lem:DJ} There is a constant $C>0$ such that
  \[
    \Bigl|\frac{\partial}{\partial y}(JF)_{(y,\theta)}^{-1}\Bigr|\le
    C\uex^{-n}, \quad \Bigl|\frac{\partial}{\partial
      \theta}(JF)_{(y,\theta)}^{-1}\Bigr|\le Cn^{-(1+\alpha)},
  \]
  for all $(y,\theta)\in Y_{n,j}$ and all $n,j$.
\end{lemma}

\begin{proof}   
  By Corollary~\ref{cor:DF}, $JF(y,\theta)=\uex^nA(y,\theta)$ where
  $A(y,\theta)\sim \uex n^{1+\alpha}$ uniformly on $Y_{n,j}$.  We
  have
  \[
  \begin{split}
    \frac{\partial}{\partial y}(JF(y,\theta))^{-1}
    & =\uex^{-n}(A(y,\theta))^{-2}\frac{\partial}{\partial y}A(y,\theta)
    \\ & =\uex^{-n}A(y,\theta)^{-1}\frac{\partial}{\partial y}\log A(y,\theta)
         \sim \uex^{-(n+1)}n^{-(1+\alpha)}\frac{\partial}{\partial y}\log A(y,\theta).
       \end{split}
       \]
  Similarly,
  $ \frac{\partial}{\partial \theta}(JF(y,\theta))^{-1} \sim
  \uex^{-(n+1)}n^{-(1+\alpha)} \frac{\partial}{\partial \theta}\log
  A(y,\theta)$.  Hence, it suffices to show that
  \begin{equation} \label{eq:A} \frac{\partial}{\partial y}\log
    A(y,\theta)\ll n^{1+\alpha}, \quad \frac{\partial}{\partial
      \theta}\log A(y,\theta)\ll \uex^n.
  \end{equation}

  Writing $z_k=f^k(y,\theta)$,
  \[
    \frac{\partial}{\partial y}\log
    A(y,\theta)=\sum_{k=1}^{n-1}\frac{\partial}{\partial y}\log
    \frac{\partial f_1}{\partial x}(z_k)
    =\sum_{k=1}^{n-1}\Bigl(\frac{\partial f_1}{\partial
      x}(z_k)\Bigr)^{-1}\frac{\partial^2 f_1}{\partial x^2}(z_k)
    \frac{\partial[f^k]_1}{\partial y}(y,\theta).
  \]
  By~\eqref{eq:f1x} and Proposition~\ref{prop:xn},
  $\frac{\partial f_1}{\partial x}(z_k)\sim 1$ and
  $\frac{\partial^2 f_1}{\partial x^2}(z_k)\ll z_k^{\gamma-1}\ll
  (n-k)^{-(1-\alpha)}$.  Moreover, it follows from
  Lemma~\ref{lem:DF} that
  \[
    \frac{\partial [f^k]_1}{\partial y}(y,\theta)\ll
    (n/(n-k))^{1+\alpha}.
  \]
  Hence
  \[
  \begin{split}
    \frac{\partial}{\partial y}\log A(y,\theta)
    & \ll \sum_{k=1}^{n-1}(n-k)^{-(1-\alpha)}n^{1+\alpha}(n-k)^{-(1+\alpha)}
      =n^{1+\alpha}\sum_{k=1}^n k^{-2}
      \ll n^{1+\alpha},
    \end{split}
    \]
  establishing the first estimate in~\eqref{eq:A}.

  Proceeding similarly for the second estimate
  \begin{equation} \label{eq:A2} \frac{\partial}{\partial \theta}\log
    A(y,\theta) =\sum_{k=1}^{n-1}\Bigl(\frac{\partial f_1}{\partial
      x}(z_k)\Bigr)^{-1} \Bigl[ \frac{\partial^2 f_1}{\partial
      x^2}(z_k) \frac{\partial[f^k]_1}{\partial \theta}(y,\theta)+
    \frac{\partial^2 f_1}{\partial x\partial\theta}(z_k)
    \frac{\partial[f^k]_2}{\partial \theta}(y,\theta) \Bigr].
  \end{equation}

  Our assumptions on $f_1$ imply in particular that
  $\frac{\partial^2 f_1}{\partial x\partial\theta}$ is bounded, so
  the second term in~\eqref{eq:A2} is $O(\uex^n)$.  The calculation at
  the end of the proof of Corollary~\ref{cor:DF} shows that
  $\frac{\partial[f^k]_1}{\partial \theta}(y,\theta) \ll
  (n-k)^{-(1+\alpha)}\uex^n$.  Hence the first term in~\eqref{eq:A2}
  is bounded up to a constant by
  \[
    \sum_{k=1}^{n-1}(n-k)^{-(1-\alpha)} (n-k)^{-(1+\alpha)}\uex^n
    =\uex^n\sum_{k=1}^{n-1}k^{-2} \ll \uex^n,
  \]
  establishing the second estimate in~\eqref{eq:A}.
\end{proof}

\begin{cor} \label{cor:PE2} There is a constant $C>0$ such that
  \[
    |1/JF(y,\theta)-1/JF(y',\theta')|\le C
    \inf_a(1/JF)|F(y,\theta)-F(y',\theta')|,
  \]
  for all $(y,\theta),\,(y',\theta')\in a=Y_{n,j}$ and all $n,j$.
\end{cor}

\begin{proof}
  First, we prove the result under the simplifying assumption that $a$
  is a rectangle.  In particular, the line segments
  $[(y,\theta),(y',\theta)]$ and $[(y',\theta),(y',\theta')]$ lie in
  $a$.  By Lemma~\ref{lem:DJ} and the mean value theorem,
  $|1/JF(y,\theta)-1/JF(y',\theta)|\ll \uex^{-n}|y-y'|$.  By
  Corollary~\ref{cor:DF},
  $\uex^{-n}|y-y'|\ll
  \uex^{-n}n^{-(1+\alpha)}|F(y,\theta)-F(y',\theta)| \ll
  \inf_a(1/JF)|F(y,\theta)-F(y',\theta)|$.  Hence
  $|1/JF(y,\theta)-1/JF(y',\theta)|\ll
  \inf_a(1/JF)|F(y,\theta)-F(y',\theta)|$.  Similarly,
  $|1/JF(y',\theta)-1/JF(y',\theta')|\ll
  \inf_a(1/JF)|F(y',\theta)-F(y',\theta')|$.  The desired estimate
  follows.

  In general, Proposition~\ref{prop:xn} ensures that there is a
  constant $c_2>0$ such that the line segments lie in the union of
  partition elements $Y_{m,j}$ with $m\ge c_2n$, and the argument
  above is unaffected.
\end{proof}

Let $\alpha^Y_k$ denote the refinement of $\alpha^Y$ into $k$-cylinders.

\begin{cor} \label{cor:PEGM} There exists a constant $C>0$ such that
  $\sup_a JF^k\le C\inf_a JF^k$ for all $a\in\alpha^Y_k$, $k\ge1$.
\end{cor}

\begin{proof}
  Write $x=(y,\theta),\,x'=(y',\theta')$.  First suppose that
  $x,x'\in a$, $a\in\alpha^Y$.  By Corollary~\ref{cor:PE2}, there is a
  constant $C_1>0$ such that
  \[
    \begin{split}
    \frac{JF(x)}{JF(x')}  =1+\frac{1/JF(x')-1/JF(x)}{1/JF(x)}
    & \le 1+\frac{|1/JF(x)-1/JF(x')|}{\inf_a 1/JF} 
    \\ & \le 1+C_1|Fx-Fx'| \le e^{C_1|Fx-Fx'|}.
  \end{split}
  \]
  Hence $|\log JF(x)-\log JF(x')|\le C_1|Fx-Fx'|$.

  Now suppose that $x,x'\in a$, $a\in\alpha^Y_k$ for some $k\ge1$.  Then
\[  \begin{split}
    |\log JF^k(x)- & \log JF^k(x')|
                     \le 
                     \sum_{j=0}^{k-1}|\log JF(F^jx)-\log JF(F^jx')|
    \\ & \le C_1\sum_{j=0}^{k-1}|F^{j+1}x-F^{j+1}x'|
         \le C_1\sum_{j=0}^{k-1}\uex^{-j}|F^kx-F^kx'|\le \uexfractwo C_1\diam Y.
       \end{split}
       \]
  The result follows with $C=e^{\uexfractwo C_1\diam Y}$.
\end{proof}

\begin{cor}[Bounded distortion] \label{cor:DJ} There is a constant
  $C>0$ such that
  \[
    \sup_a |\nabla(JF)^{-1})(DF)^{-1}|JF\le C \quad\text{for all $a\in\alpha^Y$}.
  \]
\end{cor}

\begin{proof}
Let $a=Y_{n,j}$.
  By Corollary~\ref{cor:DF},
  \begin{equation} \label{eq:DF}
    (DF)^{-1}=\uex^{-n}A^{-1}\begin{pmatrix} \uex^n & -B\\
      0 & A \end{pmatrix} \ll
    \begin{pmatrix}
      n^{-(1+\alpha)} & n^{-(1+\alpha)} \\
      0 & \uex^{-n} \end{pmatrix}
  \end{equation}
uniformly on $a$.
  By Lemma~\ref{lem:DJ},
  \[
    |(\nabla(JF)^{-1})(DF^{-1})| \ll
    \begin{pmatrix} \uex^{-n} & n^{-(1+\alpha)}
    \end{pmatrix}
    \begin{pmatrix}
      n^{-(1+\alpha)} & n^{-(1+\alpha)} \\
      0 & \uex^{-n} \end{pmatrix}
  \ll \uex^{-n}n^{-(1+\alpha)}
  \]
  on $a$.  Finally, apply Corollary~\ref{cor:DF}.
\end{proof}

\section{Mixing properties of $f$ and $F$}
\label{sec:mix}

In this section, we show that the first return map $F:Y\to Y$ has a unique absolutely continuous
invariant probability measure $\mu_Y$ and that $F$ is mixing.
We also show that the underlying map $f:X\to X$ has a unique (up to scaling) absolutely
continuous invariant $\sigma$-finite measure $\mu$.  When
$\gamma<1$, this is a finite measure and it is mixing.
These results are obtained in Subsection~\ref{subsec:mix}.
In the process of obtaining these results we show that $f$ is modelled
by a Young tower with polynomial tails.  A logarithmic factor in this tail rate is removed in Subsection~\ref{sec:SV}. This is already sufficient to obtain many of the results announced in the Introduction, as explained in Subsection~\ref{sec:stat}.
In Subsection~\ref{sec:ap}, we obtain an aperiodicity property for
$F$.

Recall that the first return map $F:Y\to Y$ is topologically mixing
with finite images, and has bounded distortion.  If in addition $F$
were Markov, then the results in this section would be easier to
deduce from standard results.
Our strategy is to further induce \(F\), with exponential tails, to a
full-branched Gibbs-Markov map $G:Z\to Z$ as follows: 

\begin{lemma} \label{lem:M} There exists a refinement $\alpha^Y_1$ of the
  partition $\alpha^Y$ for $F:Y\to Y$, an open set \(Z \subset Y\)
  consisting of a union of elements of \(\alpha^Y_1\) and a map
  $\rho:Z \to\Z^+$ constant on elements of $\alpha^Z=\{a\in\alpha^Y_1:a\subset Z\}$ such that
  \begin{enumerate}[label=(\alph*)]
  \item $G=F^\rho:Z\to Z$ is a full-branched Gibbs-Markov map
with partition $\alpha^Z$.
  \item \(\Leb(\rho>k)=O(\delta^k)\) for some \(\delta \in (0,1)\).
\item $\gcd\{n\ge1:\{\varphi=n,\,\rho=1\}\neq\emptyset\}=1$.
 \item There exists $n\ge1$ such that 
 $F\big(Z\cap\Int\{\varphi=n\}\big)=Y$.
  \end{enumerate}
\end{lemma}

We postpone the proof of Lemma~\ref{lem:M} to Appendix~\ref{app-tower}.
Parts (a) and (b) can be proven in the general setting of piecewise
expanding maps, but parts (c) and (d) are specific to our map \(F\). 
Part (c) is used to prove that $F$ and $f$ are mixing in Lemmas~\ref{lem:Fmix} and~\ref{lem:fmix} respectively.
Part~(d) is
used in the proof of Lemma \ref{lem:Fmix} to prove that the invariant density for $F$ is bounded below.  

By~\cite[Theorem~4.7.4]{Aaronson},
there exists a unique absolutely continuous $G$-invariant probability measure $\mu_Z$ on $Z$.  Moreover, $\mu_Z$ is mixing and the density
\(h_Z=d\mu_Z/d\Leb\) is bounded above and below on $Z$.

\subsection{Densities and mixing} \label{subsec:mix}

In this subsection, we study the mixing properties of $f$ and $F$, the existence and uniqueness of absolutely continuous invariant measures,
and the boundedness properties of the corresponding densities.

\begin{lemma} \label{lem:Fmix} There exists a unique absolutely
  continuous $F$-invariant probability measure~$\mu_Y$.  The density
  $h_Y=d\mu_Y/d\Leb$ is bounded above and below and $F$ is mixing.
\end{lemma}

\begin{proof}
Let $G=F^\rho:Z\to Z$ be the full-branched Gibbs-Markov map on $Z\subset Y$ obtained in Lemma~\ref{lem:M}, with ergodic invariant probability measure $\mu_Z$.  
  By Lemma~\ref{lem:M}(b), \(\bar\rho=\int_Z\rho\,d\mu_Z< \infty\). 

  Form the Young tower $g:\Delta\to\Delta$ where
  \[
    \Delta=\{(z,\ell)\in Z\times\Z: 0\le \ell\le\rho(z)-1\}, \qquad
    g(z,\ell)=\begin{cases} (z,\ell+1) & \ell\le\rho(z)-2 \\ (Gz,0) &
      \ell=\rho(z)-1 \end{cases}.
  \]
  The measure $\mu_\Delta=(\mu_Z\times{\rm counting})/\bar\rho$ is an
  ergodic $g$-invariant probability measure on $\Delta$.  The
  projection $\pi:\Delta\to Y$, $\pi(z,\ell)=F^\ell z$ defines a
  semiconjugacy between $g$ and $F$, and $\mu_Y=\pi_*\mu_\Delta$ is an
  absolutely continuous $F$-invariant probability measure on
  \(Y\). Since $G$ is full-branch and \(\gcd(\rho(a):a\in\alpha^Z)=1\) by
Lemma~\ref{lem:M}(c), it follows from~\cite[Theorem~1]{Young99} that
 \(\mu_Y\) is mixing.

Next, for $E\subset Z$ measurable,
\begin{align} \label{eq:E}
\mu_Y(E)=\mu_\Delta(\pi^{-1}E)
    & =(1/\bar\rho)\int_Z\sum_{\ell=0}^{\rho-1}1_E\circ F^\ell\,d\mu_Z
    \\ & \ge (1/\bar\rho)\int_Z1_E\,d\mu_Z = (1/\bar\rho)\mu_Z(E).
\nonumber
  \end{align}
It follows that $h_Y\ge (1/\bar\rho)h_Z$ on $Z$.
Moreover, letting $n\ge1$ be as in Lemma~\ref{lem:M}(d), for any $y\in Y$ there exists $z\in Z\cap\Int\{\varphi=n\}$ with $Fz=y$.
Since $f^n$ has finitely many continuous branches, $M=|Jf^n|_\infty<\infty$.  
We obtain
\begin{align*}
h_Y(y)  =\sum_{Fy'=y}JF(y')^{-1}h_Y(y') & \ge JF(z)^{-1}h_Y(z) 
  \\ & = Jf^n(z)^{-1}h_Y(z) \ge \bar\rho^{-1}M^{-1}\inf h_Z>0.
\end{align*}
Hence $h_Y$ is bounded below.  Uniqueness of $h_Y$ follows.

  It remains to show that $h_Y$ is bounded above. This follows from a
  result of Rychlik \cite[Theorem~1]{RychlikPhD} once we check three
  conditions:
  \begin{enumerate}
  \item There exists a constant $C>0$ such that
    $\sup_a JF^k \leq C\inf_a JF^k$ for all $a\in\alpha^Y_k$, $k\ge1$.
  \item There exists $\eps>0$, $r \in (0,1)$ such that if
    $a\in\alpha^Y_k$ for some $k\ge1$ and \mbox{$\Leb(F^ka)<\eps$}, then
    $\sum_{\{a'\in\alpha^Y: \Leb(a'\cap F^ka)>0\}} \sup_{a'} 1/JF \leq
    r$.
  \item $\sum_{a\in\alpha^Y} \sup_a 1/JF < \infty$.
  \end{enumerate}
  By~\cite[Theorem~1]{RychlikPhD}, there exists an $F$-invariant
  density $h_1 \in L^\infty(Y)$. Since $h_Y$ is the unique $F$-invariant
  density, we have $h_Y = h_1$ bounded.

  Now, condition~1 holds by Corollary~\ref{cor:PEGM}.  Condition~2 is
  trivially satisfied since the set $\{F^ka:a\in\alpha^Y_k,\,k\ge1\}$ is
  finite.  By Corollary~\ref{cor:DF},
  $1/JF \sim \uex^{-{n+1}} n^{-(1+\alpha)}$ uniformly on
  $a=Y_{n,j},\,j=1,\dots,\uex^n$ as $n\to\infty$, and the third
  condition follows.
\end{proof}

Define \(\tau=\varphi_\rho=\sum_{\ell=0}^{\rho-1}\varphi\circ F^\ell:Z\to\Z^+\).

\begin{prop} \label{prop:tau} $\tau$ is Lebesgue integrable (equivalently
$\int_Z\tau\,d\mu_Z<\infty$) if and
  only if $\gamma<1$.  
\end{prop}

\begin{proof}
A standard argument, see for instance~\cite{ChernovZhang05,Markarian04}, shows that $\tau=\varphi_\rho$ satisfies $\mu_Z(\tau>n)=O((\log n)^{\alpha}n^{-\alpha})$.
Integrability for $\gamma<1$ follows.

Similarly, $\mu_Z(\tau>n)\gg (\log n)^{-1}n^{-\alpha}$
(see for example~\cite[Proposition~5.1(b)]{BMTsub}) proving nonintegrability for $\gamma\ge1$.
\end{proof}

\begin{lemma} \label{lem:fmix} There exists a unique (up to scaling)
  absolutely continuous $f$-invariant $\sigma$-finite measure~$\mu$.
  Moreover, the density $h_X=d\mu/d\Leb$ is bounded below.

  The measure $\mu$ is finite if and only if $\gamma<1$, in which
  case $f$ is mixing.
\end{lemma}

\begin{proof}
  Since $F=f^\varphi:Y\to Y$ and $G=F^\rho:Z\to Z$, it follows that
  $G=f^\tau:Z\to Z$.

  We proceed similarly to the proof of Lemma~\ref{lem:Fmix} but with
  $\rho$ replaced by $\tau$ and $F$ replaced by $f$.  Form the new
  Young tower $\tilde g:\tilde \Delta\to\tilde \Delta$,
  \[
    \tilde \Delta=\{(y,\ell)\in Z\times\Z: 0\le \ell\le\tau(y)-1\},
    \qquad \tilde g(y,\ell)=\begin{cases} (y,\ell+1) &
      \ell\le\tau(y)-2 \\ (Gy,0) & \ell=\tau(y)-1 \end{cases}.
  \]
  The ergodic $\tilde g$-invariant measure $\mu_Z\times{\rm counting}$
  is finite if and only if
  \mbox{$\bar\tau=\int_Z\tau\,d\mu_Z<\infty$}.  Equivalently
  $\int_Z\tau\,d\Leb<\infty$, and by Proposition~\ref{prop:tau}, this
  holds if and only if $\gamma<1$.

  When $\gamma<1$, the measure
  $\tilde \mu_\Delta=(\mu_Z\times{\rm counting})/\bar\tau$ is an
  ergodic $\tilde g$-invariant probability measure on $\tilde \Delta$.
  The projection $\tilde \pi:\tilde \Delta\to X$,
  $\tilde \pi(y,\ell)=f^\ell y$ defines a semiconjugacy between
  $\tilde g$ and $f$, and $\mu=\tilde \pi_*\tilde \mu_\Delta$ is an
  absolutely continuous $f$-invariant probability measure.  
  Lemma~\ref{lem:M}(c) implies that
  \(\gcd\{\tau(a),\,a\in\alpha^Z\}=1\).  Since $G$ is a
  full-branch Gibbs-Markov map, it follows
  from~\cite[Theorem~1]{Young99} that $\tilde \mu_\Delta$, and hence
  $\mu$, is mixing.

  Again, as in the proof of Lemma~\ref{lem:Fmix},
  $h_X\ge (1/\bar\tau)h_Z$ on $Z$.  By Lemma~\ref{lem:M}, $Z$ is open, so by Proposition~\ref{prop:ftopmix}
  there exists $n\ge1$ such that $f^nZ=X$.  
Since $f^n$ has finitely many branches, $M=|Jf^n|_\infty<\infty$.  
Given $x\in X$, choose $z\in Z$ such that $f^nz=x$.  Then
  \[
    h_X(x)= \sum_{f^nx'=x} Jf^n(x')^{-1} h_X(x')
    \ge  Jf^n(z)^{-1} h_X(z)\ge (1/\bar\tau)M^{-1}\inf h_Z>0.
  \]
  Hence $h_X$ is
  bounded below, and uniqueness of $\mu$ follows.

  When $\gamma\ge1$, we proceed in the same way but without
  normalising by $\bar\tau$.
\end{proof}

\subsection{Tail estimate for $\tau=\varphi_\rho$}
\label{sec:SV}

As noted in the proof of Proposition~\ref{prop:tau}, the induced return time 
$\tau=\varphi_\rho$ satisfies the tail estimate
$\mu(\tau>n)=O((\log n)^\alpha n^{-\alpha})$.
In this subsection, we show how to remove the logarithmic factor using ideas from
Sz\'asz and Varj\'u~\cite{SV07}.  

\begin{lemma} \label{lem:tautails}
$\mu_Z(\tau>n)=O(n^{-\alpha})$.
\end{lemma}

Following~\cite[Lemma~5.1]{ChernovZhang08} and~\cite[Lemma~16]{SV07}, the crucial ingredient for proving Lemma~\ref{lem:tautails} is the following estimate.
Fix $p,q\in(0,1)$ satisfying $p<(1-q)\alpha$.
Let
\[
  Y(k,n)=\{\varphi = n \text{ and } \varphi \circ F^k > n^{1-q}\}\subset Y.
\]

\begin{prop} \label{prop:SV}
There exists \(C>0\) such that
\[
\mu_Y(Y(k,n)) \le Cn^{-(1+\alpha+p)}
\quad\text{for all $k,n\ge1$}.
\]
\end{prop}

\begin{proof}
For $k\ge1$, denote by \(\cH^k \) the set of inverse branches $h:F^ka^h\to a^h$  of \(F^k\). 
Define $S_n=\{\varphi>n^{1-q}\}$, and notice that
\[
    Y(k,n) = \Bigg(\bigcup_{j=1}^{4^n} Y_{n,j}\Bigg) \cap F^{-k}(S_n)  
    =
    \bigcup_{j=1}^{4^n} \bigcup_{h \in \cH^k} h(S_n)\cap Y_{n,j}.
\]
But $h(S_n)$ is contained in the $k$-cylinder $a^h\in\alpha^Y_k$ while
$Y_{n,j} \in \alpha^Y$. Therefore, if $h(S_n) \cap Y_{n,j}\neq \emptyset$ then $a^h \subset Y_{n,j}$.
It follows that $\bigcup_{h \in \cH^k} h(S_n)\cap Y_{n,j}\subset \bigcup_{h \in \cH^k: a^h \subset Y_{n,j}} h(S_n)$ and so
\[
 Y(k,n)\subset \bigcup_{j=1}^{4^n} \bigcup_{h \in \cH^k\,:\, a^h \subset Y_{n,j}} h(S_n).
\]
Hence
\[
  \mu_Y(Y(k,n))\le \sum_{\{a \in \alpha^Y\,:\,\varphi(a) =n\}} \sum_{\{h \in
    \cH^k\,:\, a^h \subset a
  \}} \mu_Y( h(S_n)).
\]

If $a^h\subset a$, then \(h=
  h_a\circ \tilde h\), where $h_a\in\cH^1$ and $\tilde h\in\cH^{k-1}$ are inverse branches with \(h_a:Fa \to a\).
By Lemma~\ref{lem:Fmix}, the density $d\mu_Y/d\Leb$ is bounded above and below.  Using this and $F$-invariance of $\mu_Y$,
\[
\begin{split}
\sum_{\{h \in \cH^k\,:\, a^h \subset a \}}  \mu_Y( h(S_n))
 & \ll |Jh_a|_\infty \sum_{\tilde h \in \cH^{k-1}}\mu_Y(\tilde h(S_n)) \\
& = |Jh_a|_\infty \,\mu_Y(F^{-(k-1)}(S_n)) 
= |Jh_a|_\infty \,\mu_Y(S_n).
\end{split}
\]
By Corollaries~\ref{cor:Leb} and~\ref{cor:DF},
\[
\mu_Y(Y(k,n))\ll  \mu_Y(S_n) \sum_{\{a \in \alpha^Y\,:\,\varphi(a) =n\}}|Jh_a|_\infty 
 \ll 
   n^{-(1-q)\alpha}n^{-(1+\alpha)}
\ll n^{-(1+\alpha+p)}
\]
by the choice of $p$ and $q$.
\end{proof}

Lemma~\ref{lem:tautails} now holds by standard arguments.  We follow the exposition in~\cite{BBM19}.
Define
\[
Z_b(n)=\big\{\rho\le b\log n \text{ and } \max_{0\le\ell<\rho}\varphi\circ F^\ell \le\tfrac12 n \text{ and } \tau\ge n\big\}\subset Z.
\]

\begin{cor} \label{cor:SV}
Let $b>0$.  Then $\mu_Z(Z_b(n))=o(n^{-\alpha})$.
\end{cor}

\begin{proof}
Define
\[
   Y_b(n)=\{\varphi = n \text{ and } \varphi \circ F^k > n^{1-q} \text{ for some } 1\le k\le 2b \log n\}\subset Y.
\]
By Proposition~\ref{prop:SV},
\[
\mu_Y(Y_b(n))\le C(2b\log n)n^{-(1+\alpha+p)} \ll n^{-(1+\alpha+p/2)}.
\]

Let $z\in Z_b(n)$.  Define $\varphi_1(z)=\max_{0\le\ell<\rho(z)}\varphi(F^\ell z)$
and choose $\ell_1(z)\in\{0,\dots,\rho(z)-1\}$ such that
$\varphi_1(z)=\varphi(F^{\ell_1(z)}z)$.
Define $\varphi_2(z)=\max_{0\le \ell<\rho(z)\,,\,\ell\neq\ell_1(z)}\varphi(F^\ell z)$.

Now, $n\le\tau\le\varphi_1+(\rho-1)\varphi_2\le \frac12 n+(b\log n)\varphi_2$.
Hence
\[
\frac{n}{2b\log n}\le \varphi_2\le \varphi_1\le \frac{n}{2}.
\]
In particular, $\varphi_1>\varphi_2^{1-q}$ and $\varphi_2>\varphi_1^{1-q}$ for $n$ large.

Choose $\ell_2(z)\in\{0,\dots,\rho(z)-1\}$ such that $\ell_2(z)\neq\ell_1(z)$ and
$\varphi_2(z)=\varphi(F^{\ell_2(z)})$.  
Suppose for definiteness that $\ell_1(z)<\ell_2(z)$ (the other case is similar). 
Let $m=\varphi_1(z)$, $k=\ell_2(z)-\ell_1(z)$.
Then 
\begin{itemize}

\parskip=-2pt
\item $\varphi(F^{\ell_1(z)}z)=\varphi_1(z)=m$;
\item $\varphi\circ F^k(F^{\ell_1(z)}z)=\varphi(F^{\ell_2(z)}z)=\varphi_2(z)>
\varphi_1(z)^{1-q}=m^{1-q}$; 
\item $1\le k\le\ell_2(z)\le b\log n\le2 b\log\varphi_1(z)=2b\log m$
for $n$ large.
\end{itemize}
Hence, $F^{\ell_1(z)}z\in Y_b(m)$ for $n$ large.
We have shown that
\[
Z_b(n)\subset F^{-\ell}Y_b(m)
\quad\text{for some $\ell<b\log n,\,m\ge n/(2b\log n)$},
\]
and so
\[
\begin{split}
\mu_Z(Z_b(n))
& \ll \mu_Y(Z_b(n))\ll \log n\sum_{m\ge n/(2b\log n)}\mu_Y(Y_b(m)) \\
& \ll \log n\sum_{m\ge n/(2b\log n)}m^{-(1+\alpha+p/2)}
\ll \log n(n/\log n)^{-(\alpha+p/2)}=o(n^{-\alpha})
\end{split}
\]
as required.
\end{proof}

\begin{pfof}{Lemma~\ref{lem:tautails}}
Let $\Delta=\{(z,\ell)\in Z\times\Z:0\le\ell<\rho(z)\}$ be the Young tower from the proof of Lemma~\ref{lem:Fmix} with probability
measure $\mu_\Delta=(\mu_Z\times{\rm counting})/\bar\rho$.
Recall that $\mu_Y=\pi_*\mu_\Delta$ where $\pi(z,\ell)=F^\ell z$.

Write $\max_{0\le \ell<\rho(z)}\varphi(F^\ell z)=\varphi(F^{\ell_1(z)}z)$ where
$\ell_1(z)\in\{0,\dots,\rho(z)-1\}$.
Then 
\[
\begin{split}
\mu_Z\big( & z\in Z  :\max_{0\le\ell<\rho(z)}\varphi(F^\ell z)>n/2\big)
=\bar\rho\mu_\Delta\{(z,0)\in\Delta:\varphi(F^{\ell_1(z)}z)>n/2\}
\\ & =\bar\rho \mu_\Delta\{(z,\ell_1(z)):\varphi(F^{\ell_1(z)}z)>n/2\}
 =\bar\rho \mu_\Delta\{(z,\ell_1(z)):\varphi\circ\pi(z,\ell_1(z))>n/2\}
 \\ &
\le \bar\rho \mu_\Delta\{p\in\Delta:\varphi\circ\pi(p)>n/2\}
= \bar\rho \mu_Y\{y\in Y:\varphi>n/2\}
=O(n^{-\alpha}).
\end{split}
\]
Hence, by Corollary~\ref{cor:SV},
\[
\mu_Z(\rho\le b\log n \text{ and }\tau> n)=O(n^{-\alpha}).
\]
Finally, by Lemma~\ref{lem:M}(a),
 $\mu_Z(\rho>b\log n)=O(\delta^{b\log n})=O(n^{b\log \delta})=o(n^{-\alpha})$
for any $b$ fixed sufficiently large.  Hence
$\mu_Z(\tau>n)=O(n^{-\alpha})$ as required.
\end{pfof}

\subsection{Proof of upper bounds for decay of correlations, and various statistical properties}
\label{sec:stat}

We suppose throughout this subsection that $\gamma<1$, and set $\alpha=1/\gamma$.
In the proof of Lemma~\ref{lem:fmix}, we showed that 
the intermittent map $f:X\to X$ is modelled by a Young tower $\tilde g:\tilde\Delta\to\tilde\Delta$ with
first return $G=f^\tau:Z\to Z$.
By Lemma~\ref{lem:tautails}, the return time tails satisfy
$\mu_Z(\tau>n)=O(n^{-\alpha})$.
Accordingly we can read off
numerous statistical properties that hold for all H\"older (and dynamically H\"older) observables $v:X\to\R$.  

Recall that $\tilde\pi:\tilde\Delta\to X$, given by $\tilde\pi(z,\ell)=f^\ell z$ is a semiconjugacy between $\tilde g$ and~$f$.  Moreover, we have invariant ergodic probability measures $\tilde\mu_\Delta$ on $\tilde\Delta$ and $\mu$ on $X$ where
$\tilde\mu_\Delta=(\mu_Z\times{\rm counting})/\bar\tau$ and
$\mu=\tilde\pi_*\mu_\Delta$.

Recall from Lemma~\ref{lem:M} that $G:Z\to Z$ is a full-branched Gibbs-Markov map with partition $\alpha^Z$.
Define the separation time $s(z,z')$ on $Z$ to be the least integer $n\ge0$ such that $G^nz,G^nz'$ lie in distinct partition elements in $\alpha^Z$.
For $\theta\in(0,1)$ and $v:X\to\R$, define
\[
\|v\|_{\cH_\theta}=|v|_\infty+|v|_{\cH_\theta},
\quad
|v|_{\cH_\theta}=\sup_{z,z'\in Z:z\neq z'}\sup_{0\le\ell<\tau(z)}
\frac{|v(f^\ell z)-v(f^\ell z')|}{\theta^{s(z,z')}}.
\]
An observable $v:X\to\R$ is said to be {\em dynamically-H\"older} if $\|v\|_{\cH_\theta}<\infty$ for some choice of $\theta$.

\begin{prop} \label{prop:H}
H\"older observables are dynamically H\"older.
Moreover, for $v\in C^\eta(X)$, $\eta\in(0,1)$,
we have $\|v\|_{\cH_\theta}\le 2^{\eta/2}\theta^{-1}\|v\|_\eta$
where $\theta=4^{-\eta}$.
\end{prop}

\begin{proof}
Let $z,z'\in Z$ with $s(z,z')=n$.  Then
\[
\sqrt 2 =\diam M\ge |G^nz-G^nz'|\ge 4^n|z-z'|,
\]
so $|z-z'|\le \sqrt 2\, 4^{-s(z,z')}$ for all $z,z'\in Z$.

Now, let $z,z'\in Z$, $0\le\ell<\tau(z)$.  
Then
\begin{align*}
|v(f^\ell z-v(f^\ell z')| & \le \|v\|_\eta \,|f^\ell z-f^\ell z'|^\eta
\le \|v\|_\eta \,|Gz-Gz'|^\eta
\\ & \le \|v\|_\eta 2^{\eta/2} \,\theta^{s(Gz,Gz')}
\le \|v\|_\eta 2^{\eta/2}\theta^{-1} \,\theta^{s(z,z')}
\end{align*}
yielding the desired estimate.
\end{proof}

\begin{pfof}{Theorems~\ref{thm:decay}(a),~\ref{thm:stats} and~\ref{thm:LD}}
(The proof of Theorem~\ref{thm:LD}(c) for $\gamma\in(\frac12,1)$ is momentarily contingent on Theorem~\ref{thm:stab}.)

Given $v:X\to\R$, we define the lifted observable $\tilde v=v\circ\tilde\pi:\tilde\Delta\to\R$.  Since $\tilde\pi$ is a measure-preserving semiconjugacy, the desired statistical properties for $v$ follow from those for $\tilde v$.  
Also, 
$|v|_{\cH_\theta}=
\sup_{z,z'\in Z:z\neq z'}\sup_{0\le\ell<\tau(z)}
\frac{|\tilde v(z,\ell)-\tilde v(z',\ell)|}{\theta^{s(z,z')}}$,
so dynamically H\"older observables lie in the standard function space $C_\theta(\tilde\Delta)$ considered on one-sided Young towers~\cite{Young99}.
The upper bound on decay of correlations in Theorem~\ref{thm:decay}(a) now follows from~\cite[Theorem~3]{Young99}.  

Next, we turn to Theorem~\ref{thm:stats}.
Part~(a) holds by~\cite[Theorem~4]{Young99}.
Parts~(b) and~(c) follow respectively from~\cite[Theorems~1.3 and~1.2]{Gouezel05}.
Part~(f) is proved in~\cite[Theorem~5.3]{CunyDedeckerKorepanovMerlevedesub} and
part~(d) is a standard consequence.
Part~(e) is proved in~\cite[Theorem~2.2]{AntoniouMapp}.

Finally, we consider Theorem~\ref{thm:LD}.
Part~(a) follows from Theorem~\ref{thm:decay}(a) by~\cite[Theorem~1.2]{M09}.
Part~(b) is proved in~\cite[Theorem~1.4]{GouezelM14}.
Part~(c) is an immediate consequence of part~(b) together with the CLT for
$\gamma<\frac12$ and Theorem~\ref{thm:stab} for $\gamma\in(\frac12,1)$.
\end{pfof}

\begin{rmk}
Alternative references for some of the results in the above proof include~\cite{CunyMerlevede15,DedeckerMerlevede15,Korepanov18,KKM18,MN05,MN08,MTorok12}.

For brevity, we have omitted various other statistical properties that follow from the existence of the Young tower $\tilde\Delta$ such as  concentration inequalities~\cite{GouezelM14}.  Also, for $\gamma<\frac12$,
homogenization (convergence of fast-slow systems to a stochastic differential equation) holds when the fast dynamics is given by $f$, see~\cite{CFKMZ,CFKMZsub,GM13,KM16,KM17,KKMsub}.  
\end{rmk}

\subsection{Aperiodicity}
\label{sec:ap}

Let $S^1=\{\omega\in\C:|\omega|=1\}$ and consider the cohomological equation
\begin{equation} \label{eq:cohom} v\circ F=\omega^\varphi v,
\end{equation}
where $v:Y\to S^1$ is measurable and $\omega\in S^1$.  If $\omega=1$, then since
$F$ is ergodic, the measurable solutions to equation~\eqref{eq:cohom}
are precisely the constant solutions.  Absence of solutions for
$\omega\ne1$ is called {\em aperiodicity}.  In this subsection, we prove:

\begin{lemma} \label{lem:aperiodic} For each
  $\omega\in S^1\setminus\{1\}$ there are no measurable solutions
  $v:Y\to S^1$ to equation~\eqref{eq:cohom}.
\end{lemma}

Aperiodicity is useful for ruling out peripheral spectra for certain
twisted transfer operators.  Instances of this are seen in
Corollaries~\ref{cor:LY}(e) and~\ref{cor:specR}(ii) below.

For the moment, consider an  arbitrary
ergodic measure-preserving transformations $F:Y\to Y$ defined on a
probability space $(Y,\mu_Y)$.  Let $U:L^1(Y)\to L^1(Y)$ denote the
Koopman operator $Uv=v\circ F$ and define the transfer operator
$R:L^1(Y)\to L^1(Y)$, where $\int_Y Rv\,w\,d\mu_Y=\int_Y v\,Uw\,d\mu_Y$
for all $v\in L^1(Y)$, $w\in L^\infty(Y)$.

For $\omega\in S^1$, we define the twisted Koopman and transfer operators
$U(\omega)v=\bar \omega^{\varphi}Uv =\bar \omega^{\varphi}v\circ F$ and
$R(\omega)v=R(\omega^\varphi v)$.  Note that $R(\omega)$ is the $L^2$ adjoint of
$U(\omega)$ but that $R(\omega):L^1\to L^1$ is the dual of
$U(\bar \omega):L^\infty\to L^\infty$.  (This discrepancy between adjoints
and duals over the complex numbers is standard.)

\begin{prop} \label{prop:aperiodic}
  Suppose that $F:Y\to Y$ is ergodic.  Let $\omega\in S^1$ and let
  $v:Y\to\C$ be $L^1$.  Then $U(\omega)v=v$ if and only if $R(\omega)v=v$, in
  which case $|v|$ is constant.
\end{prop}

\begin{proof}
  First, note that if $U(\omega)v=v$, then $|v|\circ F=|v|$ and so $|v|$ is
  constant by ergodicity.

  Next, recall that $RU=I$ and hence $R(\omega)U(\omega)=I$ for all $\omega$.  If
  $U(\omega)v=v$, then $v=R(\omega)U(\omega)v=R(\omega)v$, proving one direction.

  Conversely, suppose that $R(\omega)v=v$.  By duality,
  $\int_Y v\,U(\bar \omega)^nw\,d\mu_Y=\int_Y vw\,d\mu_Y$ for every
  $w\in L^\infty$ and $n\ge1$.
  We claim that $v$ is bounded and $|v|_\infty\le |v|_1$.
  Suppose the claim is false.  Then there is a set $E$ of positive
  measure and $c>|v|_1$ such that $|v|\ge c$ on $E$.  Choose
  $w=1_E\bar v/|v|$ on $\{v\neq0\}$ and $w=1_E$ elsewhere.  Then
  \[
    \begin{split}
      c\mu_Y(E) & \le \int_E|v|\,d\mu_Y=\int_Y vw\,d\mu_Y=\int_Y v \frac1n
      \sum_{j=0}^{n-1}U(\bar \omega)^jw\,d\mu_Y \\ & =\int_Y v\frac1n
      \sum_{j=0}^{n-1}\omega^{\varphi_j}\, w\circ F^j\,d\mu_Y \le \int_Y
      |v|\frac1n \sum_{j=0}^{n-1}|w|\circ F^j\,d\mu_Y.
    \end{split}
  \]
  The last integrand is dominated by $|w|_\infty |v|\in L^1$ and
  converges a.e.\ to $|v|\int_Y|w|\,d\mu_Y$ by the pointwise ergodic
  theorem.  By the dominated convergence theorem,
  \[
    \lim_{n\to\infty}\int_Y |v|\frac1n \sum_{j=0}^{n-1}|w|\circ
    F^j\,d\mu_Y = \int_Y |v|\,d\mu_Y\; \mu_Y(E).
  \]
  Hence $c\le \int_Y|v|\,d\mu_Y$ which is a contradiction.

  This proves the claim, so $v$ is bounded. In particular, $v\in L^2$
  and a computation using that $R(\omega)=U(\omega)^*$ and $R(\omega)v=v$ shows that
  $\langle U(\omega)v-v,U(\omega)v-v\rangle=0$ so that $U(\omega)v=v$ as required.
\end{proof}

Returning to the intermittent maps~\eqref{eq:EMV}, we obtain
\begin{cor} \label{cor:aperiodic} For each
  $\omega\in S^1\setminus\{1\}$ there are no $L^1$ functions $v:Y\to S^1$
  such that $R(\omega)v=v$. 
\end{cor}

\begin{proof}
This is immediate from Lemma~\ref{lem:aperiodic} and
Proposition~\ref{prop:aperiodic}.
\end{proof}

To prove Lemma~\ref{lem:aperiodic}, we make use of two Young towers
$g:\Delta\to\Delta$ and $\tilde g:\tilde \Delta\to\tilde\Delta$.  The
second of these coincides with the tower in the proof of
Lemma~\ref{lem:fmix}.  The first tower is different from those
considered so far in this paper (in particular, that of
Lemma \ref{lem:Fmix}), and is defined as follows:
\[
  \Delta=\{(y,\ell)\in Y\times\Z: 0\le \ell\le\varphi(y)-1\}, \qquad
  g(y,\ell)=\begin{cases} (y,\ell+1) & \ell\le\varphi(y)-2 \\ (Fy,0) &
    \ell=\varphi(y)-1 \end{cases}.
\]

As in
Lemma~\ref{lem:fmix}, we have ergodic $g$-invariant and
$\tilde g$-invariant
measures $\mu_\Delta=\mu_Y\times{\rm counting}$ and
$\tilde\mu_\Delta=(\mu_Z\times{\rm counting})/\bar\sigma$ on
$\Delta$ and $\tilde \Delta$.
(When $\gamma<1$, 
it follows from Corollary~\ref{cor:Leb} and
Lemma~\ref{lem:Fmix} that $\bar\varphi=\int_Y\varphi\,d\mu_Y<\infty$ and
hence we can normalize further by $\bar\varphi$ to obtain probability measures
$\mu_\Delta$ and $\tilde\mu_\Delta$.)

\begin{rmk}
  A standard strategy, used below, to establish aperiodicity is to
  show that $g$ is weak mixing.  This is made complicated by that fact
  that $g$ is nonMarkov, so we pass to the Markov extension
  $\tilde g$.  (This is similar in spirit, though the notation is more
  complicated, to the derivation of mixing properties for $F$ from
  mixing properties for $G$ in Subsection~\ref{subsec:mix}.)
\end{rmk}

Recall that \(\tau: Z \to \Z^+\),
$\tau=\varphi_\rho=\sum_{i=0}^{\rho-1}\varphi\circ F^i$.  Write
$\varphi_j=\sum_{i=0}^{j-1}\varphi\circ F^i$.  Any element of
$\tilde\Delta$ can be written uniquely as $(z,\varphi_j(z)+\ell)$
where $0\le j\le \rho(z)-1$ and $0\le \ell\le\varphi(F^jz)-1$, valid
for \(z \in Z\).

Define
\[
  \pi_\Delta:\tilde\Delta\to\Delta, \qquad
  \pi_\Delta(z,\varphi_j(z)+\ell)=(F^jz,\ell).
\]

\begin{prop} \label{prop:Delta} 
  $\pi_\Delta$ is a measure-preserving semiconjugacy from $\tilde g$
  to $g$.
\end{prop}

\begin{proof}
  To verify that $\pi_\Delta$ is a semiconjugacy ($\pi_\Delta\circ\tilde g=g\circ\pi_\Delta$), we show that
  $\pi_\Delta\circ\tilde g(z,\varphi_j(z)+\ell)=
  g\circ\pi_\Delta(z,\varphi_j(z)+\ell)$ for all $z,j,\ell$.

  Now,
  \[
    \begin{split}
      g\circ\pi_\Delta(z,\varphi_j(z)+\ell) &
      = g(F^jz,\ell)=\begin{cases} (F^jz,\ell+1) & \ell\le\varphi(F^jz)-2 \\
        (F^{j+1}z,0) & \ell=\varphi(F^jz)-1. \end{cases}
    \end{split}
  \]
  Also,
  \[
    \begin{split}
      \pi_\Delta\circ\tilde g(z,\varphi_j(z)+\ell) & =
      \begin{cases} \pi_\Delta(z,\varphi_j(z)+\ell+1) & \ell\le\varphi(F^jz)-2 \\
        \pi_\Delta(z,\varphi_{j+1}(z)) & \ell=\varphi(F^jz)-1,\,j\le\rho(z)-2 \\
        \pi_\Delta(Gz,0) &
        \ell=\varphi(F^jz)-1,\,j=\rho(z)-1 \end{cases}
      \\
      & =
      \begin{cases} (F^jz,\ell+1) & \ell\le\varphi(F^jz)-2 \\
        (F^{j+1}z,0) & \ell=\varphi(F^jz)-1. \end{cases}
    \end{split}
  \]
Hence $\pi_\Delta$ is a semiconjugacy.

  It remains to show that 
  $(\pi_\Delta)_*\tilde\mu_\Delta=\mu_\Delta$.  It suffices to test this for sets $E\times\{\ell\}$ where $E$ is a measurable subset of a partition element
$a\subset Z$, $a\in\alpha^Z$, and $0\le\ell\le\varphi(a)-1$.
By~\eqref{eq:E},
$\mu_\Delta(E\times\{\ell\})=\mu_Y(E)
     =(1/\bar\rho)\int_Z\sum_{j=0}^{\rho-1}1_E\circ F^j\,d\mu_Z$.
On the other hand 
\begin{align*}
(\pi_\Delta)_*\tilde\mu_\Delta(E\times\{\ell\}) & =
\tilde\mu_\Delta(\pi_\Delta^{-1}(E\times\{\ell\}))
= \tilde\mu_\Delta\{(x,\varphi_j(z)+\ell):F^jz\in E,\,j<\rho(a)\}
\\ & =(1/\bar\rho)\int_Z \sum_{j=0}^{\rho-1}1_E\circ F^j\,d\mu_Z.
\end{align*}
This completes the proof.
\end{proof}

\begin{pfof}{Lemma~\ref{lem:aperiodic}}
Suppose that $u:\tilde\Delta\to S^1$ is measurable and
$u\circ\tilde g=\omega u$ for some $\omega\in S^1$.
Define $V:Z\to S^1$, $V(z)=u(z,0)$.
Then $V(Gz)=u(Gz,0)=u\circ \tilde g^{\tau(z)}(z,0)=\omega^{\tau(z)}V(z)$.
Since $G$ is a full branch Gibbs-Markov map, for every $a\in\alpha^Z$ there
exists $z_a\in a$ with $Gz_a=z_a$ and so $V(z_a)=\omega^{\tau(a)}V(z_a)$.
Hence $\omega^{\tau(a)}=1$ for all $a$.
By Lemma~\ref{lem:M}(c),
  $\gcd\{\tau(a),\,a\in\alpha^Z\}=1$ and it follows that $\omega=1$.
  In other words, $\tilde g:\tilde \Delta\to\tilde \Delta$ is weak mixing.  

  By Proposition~\ref{prop:Delta}, $g:\Delta\to\Delta$ is weak mixing.  Again
  this means that the equation
  $u\circ g = \omega u$ has no measurable solutions $u:\Delta\to S^1$
for each $\omega\in S^1\setminus\{1\}$.

  Let $\omega\in S^1\setminus\{1\}$ and suppose that $v:Y\to S^1$ is a
  measurable solution to~\eqref{eq:cohom}.  Define
  $u(y,\ell)=\omega^\ell v(y)$.  Then $u:\Delta\to S^1$ is measurable.  If
  $\ell\le \varphi(y)-2$, then
  $u\circ g(y,\ell)=u(y,\ell+1)=\omega^{\ell+1}v(y)=\omega u(y,\ell)$.  If
  $\ell= \varphi(y)-1$, then
  $u\circ
  g(y,\ell)=u(Fy,0)=v(Fy)=\omega^{\varphi(y)}v(y)=\omega^{\ell+1}v(y)=\omega u(y,\ell)$.
  This shows that $u\circ g=\omega u$ which is impossible since $g$ is weak mixing.  Hence there are
  no such measurable solutions to~\eqref{eq:cohom}.
\end{pfof}

\section{Estimates in two-dimensional BV}
\label{sec:BV}

Let $\lambda_m$ denote $m$-dimensional Lebesgue measure.
For $v\in L^1(Y)$, define the variation
\[
\Var v=\sup_{\omega}\int_{\R^2} v\divv\omega\,d\lambda_2,
\]
where the supremum is taken over all compactly supported $C^1$ test functions
$\omega:\R^2\to\R^2$ such that $|\omega|_\infty\le1$.
Let $\BV(Y)$ consist of those functions $v\in L^1(Y)$ such that
$\Var v<\infty$.  This is a Banach space with norm $\|v\|_\BV=\int_Y|v|\,d\lambda_2+\Var v$.
Recall that $C^1$ functions lie in $\BV(Y)$ and
$\Var v=\int_Y|\nabla v|\,d\lambda_2$ for such functions,
where $|\nabla v|=(|\partial v/\partial y|^2+|\partial v/\partial\theta|^2)^\frac12$.

We use the fact~\cite[Remark~2.14]{Giusti} that if $w$ is continuous on a set $U$ with Lipschitz boundary and $w$ is $C^1$ on $\Int U$, then
$\Var(1_U w)=\int_U|\nabla w|\,d\lambda_2+\int_{\partial U}|w|\,d\lambda_1$.
(The measure will often be suppressed when the meaning is clear.)

The following standard result \cite[Theorem~1.17]{Giusti} allows us to reduce to considering $C^1$ functions $v:\R^2\to\R$ in many estimates.

\begin{prop}  \label{prop:C1}
Let $v\in\BV(Y)$.  There exists a sequence of $C^1$ functions $v_n:\R^2\to\R$ such that $v_n\to v$ in $L^1(Y)$
and $\lim_{n\to\infty}\Var v_n=\Var v$. \qed
\end{prop}

\begin{cor} \label{cor:C1}
Let $A:L^1(Y)\to L^1(Y)$ be a bounded linear operator.
If $\Var(Av)\le C_1\int_Y|v|+C_2\Var v$ for all $v\in C^1$,
then 
$\Var(Av)\le C_1\int_Y|v|+C_2\Var v$ for all $v\in \BV(Y)$.
\end{cor}

\begin{proof}
Let $v\in \BV(Y)$ and choose a sequence $v_n$ as in Proposition~\ref{prop:C1}.
Let $\omega$ be a $C^1$ test function.
Since $Av_n\to Av$ in $L^1(Y)$,
\[
  \begin{split}
\int_{\R^2}Av\divv\omega & 
=\lim_{n\to\infty}\int_{\R^2}Av_n\divv\omega 
=\limsup_{n\to\infty}\int_{\R^2}Av_n\divv\omega  
\le \limsup_{n\to\infty}\Var(Av_n)
\\ &
\le \limsup_{n\to\infty}\Big(C_1\int_Y|v_n|+C_2\Var v_n\Big)
 = C_1\int_Y|v|+C_2\Var v.
\end{split}
\]
Taking the supremum over $\omega$ yields the desired result.
\end{proof}

We make some additional observations that are used in Section~\ref{sec:R}.

\begin{rmk}  \label{rmk:C1}
Returning to Proposition~\ref{prop:C1}, suppose in addition that $v\in L^\infty(Y)$.
Then the sequence $v_n$ can be chosen to have the additional property that
$\sup_Y|v_n|\le 3|v|_\infty$.
To see this we use the notation from the proof of~\cite[Theorem~1.17]{Giusti} 
where $v$ is denoted by $f$ and the approximating sequence $v_n$ is denoted by
$f_\eps=\sum_i\eta_{\eps_i} \star (f \varphi_i)$.  It is immediate from the definitions in~\cite{Giusti} that
$\sup_Y|f_\eps|\le 3\max_i \sup_Y| \eta_{\eps_i} \star (f \varphi_i) |\le 3\max_i (\int_Y|\eta_{\eps_i}|) \sup_Y | f \varphi_i |\le 3|f|_\infty$.
\end{rmk}

Let $\BV\!_\infty(Y)=\BV(Y)\cap L^\infty(Y)$
with norm $\|v\|_{\BV_\infty}=|v|_\infty+\Var v$.

\begin{cor} \label{cor:alg}
$\BV\!_\infty(Y)$ is a Banach algebra.
\end{cor}

\begin{proof}  Let $v,w\in \BV\!_\infty(Y)$.
By Remark~\ref{rmk:C1}, there exists a sequence of $C^1$ functions
$v_n$ such that $v_n\to v$ in $L^1(Y)$, $\sup_Y|v_n|\le 3|v|_\infty$
and $\Var v_n\to\Var v$.  Let $w_n$ be a similar approximating sequence for $w$.

Note that $v_nw_n$ is $C^1$ and hence lies in $\BV$, while  $\int_Y|v_nw_n-vw|\le \sup_Y|v_n|\int_Y|w_n-w|+(\int_Y|v_n-v|)|w|_\infty
\le 3|v|_\infty\int|w_n-w|+|w|_\infty\int_Y|v_n-v|\to0$ as \mbox{$n\to\infty$}.
Hence it follows from~\cite[Theorem~1.9]{Giusti} that
$\Var (vw)\le\liminf_{n\to\infty}\Var(v_nw_n)$.

Since $v_n$ and $w_n$ are $C^1$, we have
$\Var(v_n w_n)\le \sup_Y|v_n|\Var w_n +\sup_Y|w_n|\Var w_n$.
Hence
\[
\Var(vw)\le 
 \liminf_{n\to\infty} 3(|v|_\infty\Var w_n+|w|_\infty\Var v_n)
=
 3(|v|_\infty\Var w+|w|_\infty\Var v).
\]
It follows that
\[
\|vw\|_{BV_\infty}\le |vw|_\infty+3(|v|_\infty\Var w+|w|_\infty\Var v)
\le 3\|v\|_{BV_\infty} \|w\|_{BV_\infty}
\]
as required.
\end{proof}

Throughout the remainder of this section, $|v|_1$ denotes $\int_Y|v|\,d\lambda_2$.

\subsection{Boundary terms}

The primary difficulty in dealing with multidimensional BV is the occurrence 
of certain boundary terms.
Let $a\in\alpha^Y$ denote a partition element and consider the branch $F_a:a\to Fa$.
 Let $\partial F_a$ denote $F|_{\partial a}:\partial a\to\partial Fa$
 with $1$-dimensional derivative $D\partial F_a$.
For the Lasota-Yorke inequality (Subsection~\ref{sec:LY} below) given $v\in C^1$,
we are required to estimate
terms of the form 
\[
\int_{\partial a}|v| |D\partial F_a/JF_a|,
\]
relative to the BV norm $\|v\|_{\BV(Y)}$.
In one dimension,  $\BV(Y)$ is embedded in $L^\infty$ which simplifies the 
estimates considerably.   For higher dimensions, much more work is required,
see~\cite{Cowieson00,Cowieson02,GoraBoyarsky89} and references therein.

Our main results in this subsection are:

\begin{lemma} \label{lem:small}
Let $v:\R^2\to\R$ be $C^1$.
There is a constant $C_1>0$ such that 
\[
 \int_{\partial a}|v| |D\partial F_a/JF_a| \le C_1(|1_av|_1+n^{-(1+\alpha)}|1_a\nabla v|_1)
\quad\text{for all $a\in\alpha^Y$ with $\varphi(a)=n$}.
\]
\end{lemma}

\begin{lemma}\label{lem:large}
Let $v:\R^2\to\R$ be $C^1$.
Suppose that $u$ is sufficiently close to constant as in
Remark~\ref{rmk:u}.  Then there exists $\kappa_0\in(0,\uexfrac)$, and
for any $N_0\ge1$, there exists a constant $C_2>0$ such that
\[
\int_{\partial a}|v| |D\partial F_a/JF_a|
 \le  C_2|1_av|_1+ \kappa_0|1_a\nabla v|_1
\]
for all $a\in\alpha^Y$ with $\varphi(a)\le N_0$.
\end{lemma}

An immediate consequence is:
\begin{cor} \label{cor:largesmall}
Let $v:\R^2\to\R$ be $C^1$.
There exists $\kappa_0\in(0,\uexfrac)$ and
$C_3>0$ such that
\[ 
\sum_a \int_{\partial a}|v| |D\partial F_a/JF_a|
 \le  C_3\sum_a |1_av|_1+ \kappa_0 \sum_a |1_a\nabla v|_1
 =  C_3|v|_1+ \kappa_0 \Var v.
\]

\vspace{-6ex}
\qed
\end{cor}
In the remainder of this subsection, we prove Lemmas~\ref{lem:small} 
and~\ref{lem:large}.

Recall from Section~\ref{sec:partition}
that the partition elements form a `rectangular' grid $\{Y_{n,j},\,n\ge1,\,j=1\dots,\uex^n\}$ where there are infinitely many columns $\mathcal{C}_n$, $n\ge1$, bounded by `vertical' curves $\xi_n(\theta)$, $0\le\theta\le1$.
The column $\mathcal{C}_n$ is divided into $\uex^n$ partition elements $\{Y_{n,j}\}$ bounded by horizontal lines $\theta=j\uex^{-n}$, $j=0,\dots,\uex^n$.
In particular, the partition element $a=Y_{n,j}$ is given by
\[
Y_{n,j}=\{(y,\theta):\xi_n(\theta)\le y\le \xi_{n-1}(\theta),\,(j-1)\uex^{-n}\le\theta\le j\uex^{-n}\}.
\]
By Proposition~\ref{prop:xn},
\begin{equation} \label{eq:diff_xi}
(\xi_{n-1}(\theta)-\xi_n(\theta))^{-1}\ll n^{1+\alpha},
\end{equation}
uniformly in $\theta$.
Also, by Proposition~\ref{prop:xn},
\begin{equation} \label{eq:Mn}
M_n=\max\{|\xi_{n-1}'|_\infty,|\xi'_n|_\infty\}\ll n^{-(1+\alpha)}.
\end{equation}

We write $\partial a=H_a\cup V_a$ where $H_a$ is the union of the two horizontal edges and $V_a$ is the union of the two `vertical' edges.

\begin{pfof}{Lemma~\ref{lem:small}}
By Lemma~\ref{lem:Ha} and~\eqref{eq:diff_xi},~\eqref{eq:Mn},
\[
  \begin{split}
\int_{H_a}|v|
& \ll (\uex^n+M_n(\xi_{n-1}-\xi_n)^{-1})|1_av|_1 + |1_a\partial_\theta v|_1
+M_n|1_a\partial_yv|_1
\\ & \ll \uex^n|1_av|_1 + |1_a\partial_\theta v|_1
+|1_a\partial_yv|_1 \ll \uex^n|1_av|_1 + |1_a\nabla v|_1.
\end{split}
\]
On the horizontal edges,
$\partial F_a(y,\theta_0)=F_1(y,\theta_0)$ since horizontal lines are mapped to horizontal lines.
By Corollary~\ref{cor:DF}, $D\partial F_a = A$ 
and $|D\partial F_a/JF_a| = \uex^{-n}$.  Hence
\begin{equation} \label{eq:small_Ha}
\int_{H_a}|v|D\partial F_a/JF_a 
\ll |1_av|_1+\uex^{-n}|1_a\nabla v|_1.
\end{equation}

Similarly, it follows from Lemma~\ref{lem:Va} that
\[
\int_{V_a}|v|\ll n^{1+\alpha}|1_av|_1+|1_a\nabla v|_1.
\]
On the `vertical' edges,
by Corollary~\ref{cor:DF}, $|D\partial F_a|\ll\uex^n$ 
and $|D\partial F_a/JF_a| \ll n^{-(1+\alpha)}$.  Hence
\begin{equation} \label{eq:small_Va}
\int_{V_a}|v| |D\partial F_a/JF_a|
 \ll |1_av|_1+n^{-(1+\alpha)}|1_a\nabla v|_1.
\end{equation}
Combining~\eqref{eq:small_Ha} and~\eqref{eq:small_Va}, we obtain the result.
\end{pfof}

\begin{pfof}{Lemma~\ref{lem:large}}
We apply Theorem~\ref{thm:a}.  Since $u$ is nearly constant as in Remark~\ref{rmk:u}, we
have $|M|_\infty< 7/\sqrt{72}$ by Remark~\ref{rmk:u}  and hence
\[
  \sqrt2((1+|M|_\infty^2)^{1/2}+|M|_\infty) = \uex\kappa_0,
  \]
where $\kappa_0<\uexfrac$.
Hence  $\int_{\partial a}|v|\le K(a)|1_av|_1+\uex\kappa_0 |1_a\nabla v|_1$
where $K(a)$ is a constant.
The result follows since $|D\partial F_a/JF_a|\le \uexfracone$.
\end{pfof}

\begin{rmk} \label{rmk:boundary}
We can relax the artificial condition in Remark~\ref{rmk:u} by increasing the expansivity of $f$ so that $D\partial F_a/JF_a$ is sufficiently small.

Alternatively, we could consider higher iterates.
But now we have to check that the calculations in
Lemma~\ref{lem:small} remain intact. Note that Lemma~\ref{lem:M} also requires a
certain amount of expansion for \(F\) to overcome the complexity
growth of discontinuities of~\(F\).
\end{rmk}

\subsection{Lasota-Yorke inequality}
\label{sec:LY}

Let $\hF:L^1(Y)\to L^1(Y)$ be the transfer operator corresponding to $F$ relative to Lebesgue measure.
(So $\int_Y \hF v\,w\,d\lambda_2=
\int_Y v\,w\circ F\,d\lambda_2$ for all $v\in L^1$, $w\in L^\infty$.)
Then $\hF v=\sum_a 1_{Fa}(gv)\circ F_a^{-1}$
where $g=1/|\det(DF)|=(JF)^{-1}$.

Also, for $\omega\in\barD$, we consider the twisted transfer operator $\hF(\omega)$
given by $\hF(\omega)v=\hF(\omega^\varphi v)$,
so 
\[
\hF(\omega)v=\sum_a 1_{Fa}(g\omega^\varphi v)\circ F_a^{-1}
=\sum_a \omega^{\varphi(a)}1_{Fa}(gv)\circ F_a^{-1}.
\]
We note that
\begin{equation} \label{eq:L1}
|\hF(\omega)v|_1=|\hF(\omega^\varphi v)|_1\le |\omega^\varphi v|_1 \le |\omega||v|_1.
\end{equation}

\begin{lemma} \label{lem:LY}
There exist constants $C>0$ and $\kappa_1\in(0,1)$ such that
\[
\Var(\hF(\omega)v)\le |\omega|(C|v|_1+\kappa_1\Var v)\quad\text{for all $v\in\BV(Y)$, $\omega\in\barD$}.
\]
\end{lemma}

\begin{proof}   
By Corollary~\ref{cor:C1}, it suffices to prove this for $v\in C^1$.

First we consider the case $\omega=1$.
Note that $(gv)\circ F_a^{-1}$ is $C^1$ on $Fa$, and so
\[
  \begin{split} 
\Var(\hF v) & \le \sum_a \Var(1_{Fa} (gv)\circ F_a^{-1})
\\ & =\sum_a \int_{Fa}|\nabla [(gv)\circ F_a^{-1}]|\,d\lambda_2+
\sum_a \int_{\partial Fa}|(gv)\circ F_a^{-1}|\,d\lambda_1.
\end{split}
\]

Now,
\[
  \begin{split}
\int_{Fa}|\nabla & [(gv)\circ F_a^{-1}]| 
= \int_{Fa}|\nabla (gv)\circ F_a^{-1}\cdot D(F_a^{-1})|
\\ &  = \int_{Fa}|\nabla (gv)\circ F_a^{-1}\cdot (DF_a)^{-1}\circ F_a^{-1}|
 = \int_{a}|\nabla (gv)(DF_a)^{-1}|JF_a\\ &
\le  \int_a |v||(\nabla g)(DF_a)^{-1}|JF_a
+  \int_a |g||\nabla v||(DF_a)^{-1}|JF_a
\\ &
=  \int_a |v||(\nabla g)(DF_a)^{-1}|JF_a
+  \int_a |\nabla v||(DF_a)^{-1}|
\le  C_1|1_av|_1 +{\SMALL\uexfracone} |1_a\nabla v|_1,
\end{split}
\]
where $C_1=\sup_a
|(\nabla g)(DF_a)^{-1}|JF_a<\infty$ by Corollary~\ref{cor:DJ}
and we have used the fact that $|DF|\ge \uex$.
Also,
\[
  \begin{split}
\int_{\partial Fa}|(gv)\circ F_a^{-1}|
& =\int_{\partial a}|gv||D(\partial F_a)|
=\int_{\partial a}|v| |D\partial F_a/JF_a|.
\end{split}
\]
We have shown that
\[
  \begin{split} 
\Var(\hF v) &  
\le C_1|v|_1+ \uexfracone|\nabla v|_1+\sum_a 
\int_{\partial a}|v| |D\partial F_a/JF_a|.
\end{split}
\]
Applying Corollary~\ref{cor:largesmall}, we obtain
$\Var(\hF v)\le C|v|_1+\kappa_1\Var v$,
with $\kappa_1= \uexfracone+\kappa_0$.

For general $\omega$, we have an extra factor of $|\omega|^{\varphi(a)}$ throughout.
Since $|\omega|\le1$ and $\varphi\ge1$, this is bounded by $|\omega|$.
\end{proof}

\begin{cor} \label{cor:LY}
(a) 
$1$ is a simple eigenvalue for $\hF:L^1(Y)\to L^1(Y)$ 
with eigenfunction $h_Y=d\mu_Y/d\Leb$.
\\[.75ex]
(b) 
$\hF(\omega):L^1(Y)\to L^1(Y)$ has spectral radius at most $|\omega|$
for all $\omega\in\barD$.
\\[.75ex]
(c)
Let $\kappa_1\in(0,1)$ be as in Lemma~\ref{lem:LY}.  
Then $\hF(\omega):\BV(Y)\to\BV(Y)$ has 
essential spectral radius at most $\kappa_1|\omega|$ for all $\omega\in\barD$.
\\[.75ex]
(d) $h_Y\in\BV\!_\infty(Y)$.
\\[.75ex]
(e) 
Regarding $\hF(\omega)$ as operators on $\BV(Y)$, it holds that
$1$ is a simple isolated eigenvalue in $\spec\hF$ and
$1\not\in\spec\hF(\omega)$ for all $\omega\in\barD\setminus\{1\}$.
\end{cor}

\begin{proof}
(a) We have $\hF h_Y=h_Y$ so $1$ is an eigenvalue for $\hF$.
Simplicity follows from the fact that $h_Y$ is the unique invariant density
(Lemma~\ref{lem:Fmix}).
\\[.75ex]
(b) 
This is immediate by~\eqref{eq:L1}.
\\[.75ex]
(c,d) 
By~\eqref{eq:L1} and Lemma~\ref{lem:LY}, 
$\|\hF(\omega)v\|_\BV\le |\omega|\{(C+1)|v|_1+\kappa_1 \|v\|_\BV\}$.
Since the unit ball in $\BV(Y)$ is compact in $L^1(Y)$, the estimate on 
the essential spectrum radius follows from~\cite{Hennion93}.
Moreover, $1$ is an eigenvalue for $\hF:\BV(Y)\to\BV(Y)$ by~\cite{Hennion93}.  By Lemma~\ref{lem:Fmix}, the corresponding density coincides with $h_Y$,
 so $h_Y\in\BV(Y)$.
Also, $h_Y\in L^\infty(Y)$ by Lemma~\ref{lem:Fmix}.
\\[.75ex]
(e) 
By part (c), it suffices to consider the multiplicity of $1$ as an eigenvalue
for $\hF(\omega)$ acting on $L^1(Y)$.
Hence the result follows from part (a) for $\omega=1$ and part (b)
for $|\omega|<1$.
Finally, we note that
$\hF(\omega)=h_YR(\omega)h_Y^{-1}$ where
$R(\omega)$ is the normalized transfer operator corresponding to the
invariant measure $\mu_Y$.
By Corollary~\ref{cor:aperiodic}, $1$ is not an eigenvalue for
$R(\omega)$ when $\omega\in S^1\setminus\{1\}$.
Hence the same holds for
$\hF(\omega)$.
\end{proof}

\subsection{Tail of the return time function}

Let $c'>0$ be as in Proposition~\ref{prop:xn}.

\begin{prop} \label{prop:mu}  
There exists a constant $C>0$ such that
$\int_{\{\varphi=n\}}|v|\,d\Leb\le C  n^{-(1+\alpha)}\|v\|_\BV$
for all $v\in\BV(Y)$.
Moreover, for \(v\in \BV(Y)\),
\begin{equation} \label{eq:mu}
\int_{\{\varphi=n\}}v\,d\Leb
\sim {\SMALL \uexfracone}c' \int_\T v({\SMALL\uexfrac}+,\theta)\,d\theta\, n^{-(1+\alpha)},
\end{equation}
where the one-sided limit $v(\uexfrac+,\theta)=\lim_{y\to\uexfrac+}v(y,\theta)$ exists for
almost every $\theta$ and is integrable.

Taking $v$ to be the density $h_Y$, we obtain
\[
\mu_Y(\varphi=n)\sim c_2
n^{-(1+\alpha)} \quad\text{where
$c_2={\SMALL \uexfracone}c' \int_\T h_Y({\SMALL\uexfrac}+,\theta)\,d\theta$}.
\]
\end{prop}

\begin{proof}  
By \cite[p.~29]{Giusti},  \(\Var v=\int \Var^{y}v\, d\theta=
\int \Var^{\theta}v\,dy\)
where $(\Var^yv)(\theta)$ denotes the one-dimensional variation of $v(\cdot,\theta)$ in the $y$-variable, and similarly for $(\Var^\theta v)(y)$.

 Recall from Section~\ref{sec:partition} that for \(n \ge 1\), 
$\{\varphi=n\}=\{(y,\theta):y\in[y_n(\theta),y_{n-1}(\theta)],\theta\in\T\}$ and that $y_{n-1}-y_n\sim\uexfracone c' n^{-(1+\alpha)}$ uniformly in $\theta$. 
Hence for a.e. \(\theta\),
\begin{align*}
n^{1+\alpha} & \int_{\{\varphi(\cdot,\theta)=n\}}  |v(y,\theta)|\,dy 
 =
n^{1+\alpha}\int_{y_n(\theta)}^{y_{n-1} (\theta)}|v(y,\theta)|\,dy 
\\ & \quad  \le n^{1+\alpha}(y_{n-1} (\theta)-y_n(\theta)){\SMALL\sup_y} |v(y,\theta)|
 \ll \int|v(y,\theta)|\,dy+(\Var^yv)(\theta),
\end{align*}
so
\begin{align*}
n^{1+\alpha}\int_{\{\varphi=n\}}|v(y,\theta)|\,d\Leb
\ll \int_\T\Big(\int|v(y,\theta)|\,dy+(\Var^yv)(\theta)\Big)\,d\theta
= \|v\|_\BV.
\end{align*}
This completes the proof of the first statement.

Next, note that $\BV$ functions restrict
 to $\BV$ functions on almost all one-dimensional slices
 (see~\cite[Lemma~A.1, third statement]{Liverani13} which is based
 on~\cite[Section~5.10.2, Theorem 2]{EvansGariepy}.
Moreover, one-dimensional BV functions have one-sided limits.  Hence
$J(\theta)=\lim_{y\to0+}v(\uexfrac+y,\theta)$ exists a.e.\ and
is measurable (being a limit of measurable functions by Fubini's theorem).
For a.e.\ $\theta$,
\[
|J(\theta)|\le \sup_y |v(y,\theta)|\le \int |v(y,\theta)|\,dy+(\Var^yv)(\theta).
\]
Hence $J$ is integrable and
both sides of~\eqref{eq:mu} are well-defined.  

Let $A_n=n^{1+\alpha}\int_{\{\varphi=n\}}v\,d\Leb
- {\SMALL \uexfracone}c' \int_\T J(\theta)\,d\theta$.
To prove validity of~\eqref{eq:mu}, we must show that $\lim_{n\to\infty}A_n=0$.
Write
\[
A_n=\int_\T B_n(\theta)\,d\theta, \qquad 
B_n(\theta)=n^{1+\alpha}\int_{y_n(\theta)}^{y_{n-1} (\theta)}v(y,\theta)\,dy
- {\SMALL \uexfracone}c' J(\theta).
\]
We apply the dominated convergence theorem.

We have already seen that 
$B_n$ is dominated by the $L^1$ function $\int|v(y,\cdot)|\,dy+\Var^y+\uexfracone c' |J|$.
Next,
\[
B_n(\theta)  =
  \bigl\{n^{1+\alpha}(y_{n-1} (\theta)-y_n(\theta))-{\SMALL\uexfracone} c'\bigr\}J(\theta)
 +n^{1+\alpha}\int_{y_n(\theta)}^{y_{n-1} (\theta)}\big\{v(y,\theta)-J(\theta)\big\}\,dy.
\]
The first term converges to zero a.e.\ by the estimate for $y_{n-1}-y_n$.  
Also, $y_n\to\uexfrac+$, so
\begin{align*}
\Big|n^{1+\alpha}\int_{y_{n+1}(\theta)}^{y_n(\theta)}(v(y,\theta)-J(\theta))\,dy\Big| & \le n^{1+\alpha}(y_{n-1} (\theta)-y_n(\theta))\!\sup_{y\in [\uexfrac,y_{n-1} (\theta)]}|v(y,\theta)-J(\theta)|
\\ & \ll \sup_{y\in [\uexfrac,y_{n-1} (\theta)]}|v(y,\theta)-J(\theta)|\to0\;a.e.
\end{align*}
by the definition of $J(\theta)$.
Hence $B_n(\theta)\to0$ a.e.\ completing the proof of~\eqref{eq:mu}.

The estimate for $\mu_Y(\varphi>n)$ follows immediately.  
\end{proof}

\subsection{Estimates for $\|F_n\|_\BV$}
\label{sec:Fn}

Define the family of operators $\hF_n:\BV(Y)\to\BV(Y)$, $n\ge1$, given by
$\hF_nv=\hF(1_{\{\varphi=n\}}v)$.

\begin{lemma} \label{lem:Fn}
$\|\hF_n\|_\BV=O(n^{-(1+\alpha)})$.
\end{lemma}

\begin{proof}
Recall that
$\|\hF_nv\|_\BV= 
|\hF_nv|_1+ \Var(\hF_nv)$.
Now,
\begin{align*}
|\hF_nv|_1 & =
|\hF(1_{\{\varphi=n\}}v)|_1
 \le
|1_{\{\varphi=n\}}v|_1
\ll n^{-(1+\alpha)}\|v\|_\BV
\end{align*}
by Proposition~\ref{prop:mu}.

Next, we estimate $\Var(\hF_nv)$.  By Corollary~\ref{cor:C1},
it suffices to do this for $v\in C^1$.
Adapting the calculations in the proof of Lemma~\ref{lem:LY}, we have
\(
\Var(\hF_nv)  \le I_1+I_2+I_3,
\)
where
\[
  \begin{split}
I_1 & =
\sum_{\varphi(a)=n} 
|1_a(\nabla g)(DF_a)^{-1}JF|_\infty\int_a|v| ,
\\ 
I_2 & =
\sum_{\varphi(a)=n} 
|(DF_a)^{-1}|_\infty \int_a|\nabla v|,
\qquad
I_3  = 
\sum_{\varphi(a)=n} 
\int_{\partial a}|v||D\partial F_a/JF_a|.
\end{split}
\]

We consider partition elements of the form $a=Y_{n,j}$.
By Corollary~\ref{cor:DJ},
$|1_a(\nabla g)(DF_a)^{-1}JF|_\infty =O(1)$.
Hence  $I_1\ll\int_{\{\varphi=n\}}|v| \ll
n^{-(1+\alpha)}\|v\|_\BV$ by Proposition~\ref{prop:mu}.

By~\eqref{eq:DF}, $|(DF_a)^{-1}|_\infty \ll n^{-(1+\alpha)}$.
Hence $I_2\ll n^{-(1+\alpha)}\int_{\{\varphi=n\}}|\nabla v|
\le n^{-(1+\alpha)}\|v\|_\BV$.

Finally, by Lemma~\ref{lem:small} and
Proposition~\ref{prop:mu},
$I_3\ll \int_{\{\varphi=n\}}|v|+n^{-(1+\alpha)}\int_{\{\varphi=n\}}|\nabla v|\ll n^{-(1+\alpha)}\|v\|_\BV$.
\end{proof}

\section{Lower bounds on decay of correlations and infinite measure mixing}
\label{sec:R}

By Proposition~\ref{prop:mu}, the return time $\varphi$ is integrable if and only if $\gamma<1$. 
In this section, we establish lower bounds on decay of correlations
for a class of observables supported on $Y$ when $\gamma<1$.
Also, for $\gamma\ge1$, we obtain results on mixing for the infinite measure $\mu_Y$ for the same class of observables.

By Lemma~\ref{lem:Fmix}, the invariant density $h_Y=d\mu_Y/d\Leb$ is
bounded above and below, so the $L^p$ spaces with respect to $\mu_Y$
and $\Leb$ are identical and we can just write $L^p(Y)$.  We have
$\BV(Y)\subset L^2(Y)$ since the domain $Y$ is two-dimensional.

The transfer operator $R$ corresponding to the $F$-invariant measure
$\mu_Y$ is given by $Rv=h_Y^{-1}\hF(h_Yv)$.  Again, we consider
the twisted transfer operators
$R(\omega)v=R(\omega^\varphi v)=h_Y^{-1}\hF(\omega)(h_Yv)$.  These act
naturally on the Banach space $\mathcal{B}(Y)=h_Y^{-1}\BV(Y)$ which
consists of functions $v:Y\to\R$ such that $h_Yv\in\BV(Y)$ with norm
$\|v\|_\cB=\|h_Yv\|_{\BV}$.
Similarly we define $R_n:\cB\to\cB$, $n\ge1$, given by $R_nv=R(1_{\{\varphi=n\}}v)=h_Y^{-1}\hF_n(h_Yv)$.

Recall that $h_Y\in \BV\!_\infty(Y)=\BV(Y)\cap L^\infty(Y)$.

\begin{prop} \label{prop:L1}
  $\BV\!_\infty(Y)\subset\mathcal{B}(Y)\subset L^2(Y)$.
\end{prop}

\begin{proof} 
  Let $v\in\mathcal{B}(Y)$.  Then
  $\int v^2\,d\Leb\le |h_Y^{-2}|_\infty\int (h_Yv)^2\,d\Leb \ll
  \|h_Yv\|^2_\BV=\|v\|^2_\cB$ establishing the second inclusion.

  For the first inclusion, let $v\in \BV\!_\infty(Y)$.
By Corollary~\ref{cor:alg},
$h_Yv\in\BV\!_\infty(Y)\subset\BV(Y)$, so
$v\in\cB(Y)$.
\end{proof} 

\begin{cor} \label{cor:specR} Consider the operators
  $R(\omega):\mathcal{B}(Y)\to\mathcal{B}(Y)$.
  \\[.75ex]
  (i) $1$ is a simple isolated eigenvalue in the spectrum of $R$.
  \\[.75ex]
  (ii) $1\not\in\spec R(\omega)$ for all $\omega\in\barD\setminus\{1\}$.
  \\[.75ex]
  (iii) $\|R_n\|_\cB=O(n^{-(1+\alpha)})$.
\end{cor}

\begin{proof}
  Multiplication by $h_Y^{-1}$ is an isomorphism from
  $\BV(Y)\to \mathcal{B}(Y)$ that conjugates $\hF(\omega)$ to
  $R(\omega)$.  Hence $R(\omega)$ inherits properties of $\hF(\omega)$ in
  Corollary~\ref{cor:LY}.  
Similarly, $R_n$ inherits properties of $\hF_n$ in Lemma~\ref{lem:Fn}.
\end{proof}

For observables $v,w$ supported in $Y$, we can
write $\int v\,w\circ f^n\,d\mu=\int T_nv\,w\,d\mu$ where
$T_n=1_YL^n1_Y$ and $L$ is the transfer operator for $f$.
Defining $T(\omega)=\sum_{n=0}^\infty T_n\omega^n$, we recall from~\cite[Proposition~1]{Sarig02} the operator renewal equation $T(\omega)=(I-R(\omega))^{-1}$.

\begin{pfof}{Theorem~\ref{thm:decay}(b,c)}
Recall that $F$ is the first return map to $Y$ so $\mu_Y=\mu|_Y/\mu(Y)$.
By Corollary~\ref{cor:specR}, we have verified the assumptions of Gou\"ezel~\cite{Gouezel04a}.
Let $v\in\mathcal{B}(Y)$ and $w\in L^2(Y)\;$\footnote{In general, we require $w$ in
$L^2$ since Proposition~\ref{prop:L1} only gives $\mathcal{B}\subset L^2$.  For $v\in\BV\!_\infty$, we can take $w\in L^1$.}.
By~\cite[Theorem~1.1]{Gouezel04a},
\[
\Big|\rho(n)- \mu(Y)\sum_{j>n}\mu_Y(\varphi>j)\int v\,d\mu \int w\,d\mu\Big|\ll E(n)\|v\|_{\mathcal{B}}|w|_2
\]
But $\sum_{j>n}\mu_Y(\varphi>j)\sim \gamma(\alpha-1)^{-1}c_2 n^{-(\alpha-1)}$
with $c_2$ as given in Proposition~\ref{prop:mu}. 
Part~(b) follows with $c=\mu(Y)\gamma(\alpha-1)^{-1}c_2$.

Part~(c) is a consequence of~\cite[Theorem~1.2]{Gouezel04a}.
\end{pfof}

\begin{pfof}{Theorem~\ref{thm:infinite}}
By Corollary~\ref{cor:specR} we have verified the
assumptions in Gou\"ezel~\cite[Theorem~1.4]{Gouezel11} and
Melbourne \& Terhesiu~\cite[Theorem~2.1]{MT12}. 
Let $d_\gamma=\frac{1}{\pi}\sin\alpha\pi$.
Let $c_2$ be the constant in Proposition~\ref{prop:mu}
and define $c_4=\mu(Y)\gamma c_2$.
For $\gamma>1$, we obtain
\[
c_4 n^{1-\alpha}\int v\,w\circ f^n\,d\mu\sim d_\gamma \int v\,d\mu \int w\,d\mu,
\]
for all $v\in\mathcal{B}(Y)$ and $w\in L^2(Y)$.
For $\gamma=1$ the asymptotic holds with 
$d_\gamma=1$ and $n^{1-\alpha}$ replaced by $\log n$.
Part~(a) follows with $c=d_\gamma/c_4$.

Part~(b) is a consequence of~\cite[Theorem~2.2(c)]{MT12}.
\end{pfof}

\section{Convergence to stable laws and L\'evy processes}
\label{sec:levy}

In this section, we prove Theorem~\ref{thm:stab}.
Set $\alpha=\frac{1}{\gamma}\in(1,2)$ and
let $G_\alpha$ denote the totally skewed $\alpha$-stable law in Theorem~\ref{thm:stab}.
Define
$\sigma={\SMALL \uexfracone}c'\gamma \int_\T h_Y({\SMALL\uexfrac}+,\theta)\,d\theta\,\Gamma(1-\alpha)\cos{\SMALL\frac{\alpha\pi}{2}}$
where $c'$ is as in Proposition~\ref{prop:xn}.

\begin{prop} \label{prop:stable}
$n^{-1/\alpha}\sum_{j=0}^{n-1}(\varphi\circ f^j-\int_Y\varphi\,d\mu_Y)\to_d \sigma^{1/\alpha}G_\alpha$.
\end{prop}

\begin{proof}
We verify the conditions stated in Appendix~\ref{sec:stable}.
Taking $\omega=1$ in Corollary~\ref{cor:specR}, we see that $R:\cB(Y)\to\cB(Y)$ satisfies the required spectral gap condition.

Let $\psi=\varphi-\int_Y\varphi\,d\mu_Y$.  This is an $L^1$ function with mean zero.  Clearly $\psi$ is bounded below.  By Proposition~\ref{prop:mu},
$\mu_Y(\psi>x)\sim \sigma_1x^{-\alpha}$
where 
$\sigma_1={\SMALL \uexfracone}c'\gamma \int_\T h_Y({\SMALL\uexfrac}+,\theta)\,d\theta$.
Hence condition~\eqref{eq:tails} is satisfied (with $\sigma_2=0$).

Define $R_t=R(e^{it})$ for $t\in\R$.  By Section~\ref{sec:R}, $R_t$ is a bounded linear operator on $\cB(Y)$ for all $t$.
Note that $R_t=\sum_{n=1}^\infty R_n e^{int}$.  
It follows from Corollary~\ref{cor:specR}(iii) that 
$\sum_{n=1}^\infty n \|R_n\|_\cB<\infty$ so $t\mapsto R_t$ is $C^1$.
In particular, $\|R_t\|_\cB = O(|t|)$.
The result now follows from Theorem~\ref{thm:stable} (with $\beta=1$).
\end{proof}

\begin{prop} \label{prop:I}
Let $v:X\to\R$ be H\"older.  Suppose that $v(0,\theta)\equiv I$
for some $I\in\R$.  Define $V=\sum_{\ell=0}^{\varphi-1}v\circ f^\ell$.
Then $V-I\varphi\in L^p(Y)$ for some $p>\alpha$.
\end{prop}

\begin{proof}
Let $\eta\in(0,1)$ be the H\"older exponent for $v$ and suppose without loss that $\eta<\gamma$.  Set $\delta=\eta\alpha\in(0,1)$.
Since $\varphi\in L^q(Y)$ for all $q<\alpha$, it suffices to show that
$V-I\varphi=O(\varphi^{1-\delta})$.

Let $(y,\theta)\in Y_{n,j}$.
Then 
\[
V(y,\theta)-I\varphi(y,\theta)=\sum_{\ell=0}^{n-1}v(f^\ell(y,\theta))
-nI
=\sum_{\ell=0}^{n-1}(v(f^\ell(y,\theta)-v(f^\ell(0,\theta)).
\]
Write $f^\ell(y,\theta)=(y_\ell,\theta_\ell)$.
Then $f^\ell(0,\theta)=(0,\theta_\ell)$,
so
$|V(y,\theta)-I\varphi(y,\theta)|\le |v|_\eta \sum_{\ell=0}^{n-1}y_\ell^\eta$.
By Proposition~\ref{prop:xn},
$y_\ell\ll (n-\ell)^{-\alpha}$ for $\ell=0,\dots,n-1$, so
$|V-I\varphi|\ll |v|_\eta n^{1-\eta\alpha}\ll \varphi^{1-\delta}$ as required.
\end{proof}

\begin{pfof}{Theorem~\ref{thm:stab}}
First we prove the result under the additional assumption
that $v(0,\theta)$ is independent of $\theta$.  Evidently
this constant value is $I_v$, so Proposition~\ref{prop:I}
implies that $V-I_v\varphi\in L^p(Y)$ for some $p>\alpha$.
This is~\cite[condition~(3.2)]{MVsub}.
Also,~\cite[condition~(3.1)]{MVsub} follows from
Proposition~\ref{prop:stable}.
Hence convergence to the desired stable law follows from~\cite[Theorem~3.1]{MVsub} with $c=\bar\varphi^{-1/\alpha}I_v\sigma $.

Define $M_1=\max_{1\le\ell'\le\ell\le\varphi}(v_{\ell'}-v_\ell)\wedge
\max_{1\le\ell'\le\ell\le\varphi}(v_\ell-v_{\ell'})$
where $v_\ell=\sum_{j=0}^{\ell-1}v\circ f^j$.
Suppose for definiteness that $I_v>0$ (the case $I_v$ is treated similarly).
The calculation in Proposition~\ref{prop:I}
shows that $v_\ell=I_v \ell+O(\varphi^{1-\delta})$ for all
$0\le\ell\le\varphi$.  Hence
\[
0\le M_1\le \max_{1\le\ell'\le\ell\le\varphi}(v_{\ell'}-v_\ell)
= \max_{1\le\ell'\le\ell\le\varphi}I_v(\ell'-\ell)+O(\varphi^{1-\delta}).
\]
Since $I_v>0$ it follows that $M_1\ll \varphi^{1-\delta}$.
By~\cite[Proposition~3.5]{MVsub}, $n^{-1/\alpha}\max_{j\le n}M_1\circ F^j\to_p0$ on $(Y,\mu_Y)$.  
Hence convergence to the desired L\'evy process follows from~\cite[Theorem~3.2(a)]{MVsub}.

Finally, we relax the additional assumption on $v$.
Write $v=v'+v''$ where
$v''(y,\theta)=v(0,\theta)-I_v$.
We have $W_n=W_n'+W_n''$ where
\[
\SMALL W_n'(t)=n^{-1/\alpha}\sum_{j=0}^{[nt]-1}v'\circ f^j,\qquad
W_n''(t)=n^{-1/\alpha}\sum_{j=0}^{[nt]-1}v''\circ f^j.
\]
Note that $v''$, and hence $v'$, is H\"older and mean zero.
Moreover, $v'(0,\theta)\equiv I_v$, so
$W_n'\to_w W$ in $(D[0,\infty),\cM_1)$.
Also, $u(\theta)=v''(y,\theta)$ is a H\"older mean zero observable for the uniformly expanding map $f_2:\T\to\T$, so 
$n^{-1/2}\sum_{j=0}^{[nt]-1}u\circ f_2^j$ converges weakly to Brownian motion in the uniform topology
(see for example~\cite[Theorem~5]{HofbauerKeller82} which establishes the ASIP and hence the weak convergence).
Hence $n^{-1/2}\sum_{j=0}^{[nt]-1}v''\circ f^j$ converges weakly, so
$W_n''\to_w0$. The result follows.~
\end{pfof}

\appendix

\renewcommand{\thesubsection}{\Alph{section}.\arabic{subsection}}
\section{Construction of the Gibbs-Markov map $G$} \label{app-tower}

This section is devoted to the proof of Lemma \ref{lem:M}. The main step is to verify the hypotheses of
Theorem~3 of \cite{Esl19}.
This is done using Theorem \ref{thm:PE} below.

Recall that \(Y \subset \R^2\) is endowed with the Euclidean metric
\( |(y_1,\theta_1)-(y_2,\theta_2)|=((y_1-y_2)^2+(\theta_1-\theta_2)^2)^{1/2}\). 
For $x\in\R^2$ and $A\subset\R^2$, let  $\metr(x,A)=\inf_{y\in A}|x-y|$
(with $\metr(x,A)=\infty$ if $A=\emptyset$).
Given \(A \subset \R^2\), $\ve>0$, define
\[
\partial_{\ve} A=\{x \in A : \metr(x, \partial A)\le
\ve\}\subset A,
\] 
where \(\partial A\) is the boundary of \(A\) as
a subset of \(\R^2\). 

We prove that the first return map $F:Y\to Y$ satisfies the following properties:

\begin{thm} \label{thm:PE}
Let $\Lambda\in (\frac{1}{\uex},\frac13)$.  There exists $\epstail\in(0,1)$ and
$C>0$ such that the following hold:

\vspace{1ex}
\noindent {\bf Uniform expansion:}
    $|F_a^{-1}z_1-F_a^{-1}z_2| \le \Lambda|z_1-z_2|$
for all $z_1,z_2\in a$ with $|z_1-z_2|<\epstail$ and all $a\in\alpha^Y$.

\vspace{1ex}
\noindent {\bf Bounded distortion:}
  $(JF_a^{-1})(z_1) \leq e^{C |z_1-z_2|}(JF_a^{-1})(z_2)$
for all $z_1,z_2\in a$ and all $a\in\alpha^Y$.

\vspace{1ex}
\noindent {\bf Controlled complexity:}
For every open set \(I \subset Y\) with \(\diam I \le \epstail\) and
all \(\ve< \epstail\),
\[
 \sum_{a\in\alpha^Y}
\frac{\Leb(F_a^{-1} (\partial_{\ve}F(I\cap a))\setminus \partial_{\ve\Lambda}I)}{\Leb\partial_{\ve\Lambda} I} < \Lambda^{-1}-1.
\]

\vspace{1ex}
\noindent {\bf Set \(Z\):}
For all $\delta>0$ sufficiently small, there exist rectangles $Z,Z' \subset Y$  with $\Leb Z<\Leb Z'$ and $\diam Z' \le \delta$ such that

  \begin{equation} \label{eq:supsetY}
F(Z \cap \Int\{\varphi=n\}) = Y\quad\text{for sufficiently large $n$}; 
\end{equation}
and, there exists 
$a_1,a_2\in\alpha^Y$  such that
  \(a_i \subset Z\) and
  \(Fa_i \supset Z'\) for $i=1,2$ and 
\begin{equation} \label{eq:a1a2}
\varphi|_{a_2}-\varphi|_{a_1}=1.
\end{equation}
 \end{thm}

\begin{pfof}{Lemma~\ref{lem:M}}
Theorem~\ref{thm:PE} implies in particular that we have verified the hypotheses of~\cite[Theorem~3]{Esl19}.  This guarantees the existence of the desired refinement $\alpha^Y_1$, the subset $Z$ (as given in Theorem~\ref{thm:PE}), the return time $\rho:Z\to\Z^+$ constant on elements of $\alpha^Z=\{a'\in\alpha^Y_1:a'\subset Z\}$ and the induced map $G=F^\rho:Z\to Z$.  Moreover, conditions (a) and (b) of Lemma~\ref{lem:M} follow directly from~\cite[Theorem~3]{Esl19}(a),(c).

In addition,~\cite[Theorem~3]{Esl19}(b) states that 
    \(\Leb(\{\rho =1\} \cap a_i) >0\) for $i=1,2$.
This combined with~\eqref{eq:a1a2} guarantees that 
Lemma~\ref{lem:M}(c) holds.
Finally, Lemma~\ref{lem:M}(d) follows from~\eqref{eq:supsetY}.
\end{pfof}

In the next four subsections, we verify the four properties listed in Theorem~\ref{thm:PE}.

\subsection{Uniform expansion}

By Proposition~\ref{prop:expand},
$|DF_a^{-1}|\le \uexfracone$ on $Fa$ for all $a\in\alpha^Y$.

\begin{lemma} \label{lem:unifexp} 
For every \(\delta>0\) there exists
  \(\epstail>0\) such that for all $z_1, z_2 \in Fa$ with
  \(| z_1-z_2| < \epstail\) and all $a\in\alpha^Y$, 
there exists a path
  \(\gamma: [0,1] \to \R^2\) contained in \(Fa\), joining
  \(z_1\) and \(z_2\), and having length bounded by
  \((\curvlen) |z_1-z_2|\).
\end{lemma}

\begin{proof}
  The boundary of \(Fa\) is a rectangle except that its right
  boundary is a $C^1$ curve which we denote by \(\psi\). 
  Denote the line segment joining \(z_1\)
  and \(z_2\) by \(S\). If \(S\) lies in \(Fa\), then take
  \(\gamma\) to be the path corresponding to this line segment. If
  not, then \(S\) intersects the boundary of \(Fa\). Let
  \(p_1, p_2\) be the points of intersection closest to
  \(z_1, z_2\), respectively. Define \(\gamma\) to be the path
  corresponding to starting at \(z_1\), travelling on \(S\) until
  \(p_1\), then travelling on the boundary of \(Fa\) until
  \(p_2\) and then continuing on \(S\) to~\(z_2\). Since \(\psi\)
  is smooth, the length of \(\gamma\) can be made arbitrarily close to
  the length of \(S\) by choosing \(\epstail \) sufficiently
  small. (The path $\gamma$ may not be entirely contained in
  \(Fa\), but a small translation of it will be entirely inside
  \(Fa\).)
\end{proof}

Choose $\delta$ so that $\uexfracone(1+\delta)<\Lambda$, and fix $\epstail$
as in Lemma \ref{lem:unifexp}.
Let $z_1,z_2\in Fa$ with $|z_1-z_2|<\epstail$ and choose $\gamma$ as in
Lemma \ref{lem:unifexp}.

Now,
    $F_a^{-1}z_2-F_a^{-1}z_1
     = (F_a^{-1} \circ \gamma)(1) -(F_a^{-1} \circ \gamma) (0) 
= \int_0^1 D(F_a^{-1}\circ \gamma)(t) dt$,
so by Lemma \ref{lem:unifexp},
\[
  \begin{split}
    |F_a^{-1}z_1-F_a^{-1}z_2|
    &\leq \int_0^1|(DF_a^{-1}) (\gamma(t))\,
      \gamma'(t)|\,dt \leq \sup_{t \in [0,1]} |D
      F_a^{-1} (\gamma(t))| \int_0^1
    |\gamma'(t)|\,dt\\
    & \SMALL \leq \uexfracone (\curvlen) |z_1-z_2|<\Lambda|z_1-z_2|,
  \end{split}
\]
as required.

\subsection{Bounded distortion}
By Corollary \ref{cor:PE2}, there exists $C>0$ such that
  \[
    |1/JF(y_1,\theta_1)-1/JF(y_2,\theta_2)|\le C
    \inf_a(1/JF)|F(y_1,\theta_1)-F(y_2,\theta_2)|
  \]
  for all $(y_1,\theta_1),\,(y_2,\theta_2)\in a$ and all $a$.
Writing $z_1=F(y_1,\theta_1)$, 
$z_2=F(y_2,\theta_2)$, 
it follows that
\[
  \frac{(JF_a^{-1})(z_1)}{(JF_a^{-1})(z_2)} \leq 1+ C |z_1- z_2| \leq e^{C |z_1- z_2|},
\]
yielding the desired distortion condition.

\subsection{Controlled Complexity}
\label{subsec:Complexity}
For the proof of this property we need a generalization of \cite [Sublemma
C.1]{BT08}, which appears below as Proposition~\ref{BTGenLip}.

Recall that $\partial_\ve A$ is defined as a subset of $A$.
We also define $\tilde\partial_\ve A=\{x\in\R^2:\metr(x,\partial A)\le\ve\}$.
(Hence $\partial_\ve A=\tilde\partial_\ve A \cap A$).

Let us recall \cite[Sublemma C.1]{BT08} in a form that suffices for
our purposes. 
We refer to
its proof briefly at the end of the proof of Lemma \ref{BTGen}. 

\begin{lemma}[Sublemma C.1 of \cite{BT08}] \label{lem:BT} Suppose
  \(I\) is a non-empty measurable bounded subset of the plane and
  \(E\) is a straight line cutting \(I\) into left and right parts
  \(I_{l}\) and \(I_{r}\). Then for all \(\ve \ge 0\), \(0 \le \xi \le 1\), 
  \begin{equation} \label{eq:BT}
    \begin{split}
      \Leb(\{x \in I_{l}: \metr(x, E)\leq \ve\xi \}\setminus \{x \in
      I: \metr(x, \partial I) \leq \ve\})
      &\leq \\
      & \hspace{-3cm} \xi \Leb\{x \in I_{r}: \metr(x,
      \partial I)\le \ve\}.
    \end{split}
  \end{equation}  
\end{lemma}

In Proposition~\ref{BTGenLip} we generalize to the case
where \(E\) is the graph of a Lipschitz function, but first we prove a
lemma which is similar in flavour but applies to segments which may or
may not intersect \(I\).  This lemma would follow from the one above if
\(0 \le \xi \le 1\) and \(S\) (taking the place of \(E\)) were a
hyperplane in \(\R^2\) cutting through \(I\).

Given a straight line segment \(S \in \R^2\) and $x\in\R^2$, define
\(\metr^{\perp}\) as follows. Suppose
\(x \in \R^2\). If there exists a line that passes through \(x\),
intersects \(S\) and is perpendicular to \(S\), then
\(\metr^{\perp}(x, S) = \metr(x, S)\). If not, then
\(\metr^{\perp}(x, S) = \infty\). If \(S\) is the graph of a piecewise
constant function, then one can define \(\metr^{\perp}(x, S) \)
similarly.
  
\begin{lemma} \label{BTGen} Suppose \(I\) is a measurable
  bounded subset of the plane and \(S\) is a straight line segment in
  the plane.  Then for all \(\ve \ge 0\), \(\xi \ge 0\),
  \begin{equation} \label{eq:BTGen}
    \begin{split}
      \Leb(\{x \in I: \metr^{\perp}(x, S)\leq \ve\xi \}\setminus \{x
      \in I: \metr(x, \partial I) \leq \ve\})
      &\leq \\
      & \hspace{-3cm} \xi \Leb\{x \in I: \metr(x,
      \partial I)\le \ve\}.
    \end{split}
  \end{equation}  
\end{lemma}

\begin{proof}
  Fix \(\ve \ge 0\), \(\xi \ge 0\). Let 
\[
      A = \{x \in I: \metr^{\perp}(x,S)\leq 
      \ve\xi\}\setminus B, \text{ where }
      B = \{x \in I: \metr(x,
      \partial I)\le \ve\}.
  \]
  We show that \(\Leb A \le \xi \Leb B\).

  Let \(e_z\) be the line perpendicular to \(S\) at the point
  \(z \in S\) and let \(A_z= A \cap e_z\) and
  \(B_z=B \cap e_z\). Given \(\ve'>0\) and an interval
  \(J_{\ve'}\) of length \(\ve'\) inside \(e_z\), denote
  \(A_z(\ve') = A_z \cap J_{\ve'}\). Points of
  \(A_z(\ve)\) are by definition at least distance \(\ve\) from
  \(\partial I\) so there exists a translate of \(A_z(\ve)\)
  along \(e_z\) that lies in \(B_z\).  It follows from translation
  invariance of Lebesgue measure $\Leb_z$ on $e_z$ that
  \(\Leb_z A_z(\ve) \le \Leb_z B_z\). 

Let us write \(\xi = \lfloor \xi \rfloor + \{\xi\} \), where \(\{\xi\}\)
  denotes the fractional part of \(\xi\). Since \(A_z\)
  can be partitioned by
  \(\lfloor \xi \rfloor\) sets of the form \(A_z(\ve)\) plus
  one remainder set of the form \(A_z(\{\xi\}\ve)\), it
  follows that
  \begin{equation} \label{eq:temp} \Leb_z A_z \le
    \lfloor \xi \rfloor \Leb_z (B_z) + \Leb_z  A_z(\ve\{\xi\}).
  \end{equation}
Now we show that
  \(\Leb_z A_z(\ve\{\xi\}) \le \{\xi\}
  \Leb_z B_z\) bounding the second term of
  \eqref{eq:temp}. If \(z \in S \setminus I\), then
  \(A_z(\ve\{\xi\}) = \emptyset\) because \(\{\xi\} < 1\), so
  we are done. Otherwise, if \(z \in S \cap I\), the claim follows
  directly from the proof of Lemma \ref{lem:BT} given in
  \cite[p.1364]{BT08} because \(0\le \{\xi\} < 1\).

  We have proved that
  \(\Leb_z A_z \le \xi \Leb_z B_z\). 
  Integrating over \(z \in S\) with respect to Lebesgue measure on
  \(S\), we obtain \(\Leb A \le \xi \Leb B\) as required.
\end{proof}

\begin{prop} \label{BTGenLip} Suppose \(I\) is a measurable
  bounded subset of the plane and \(E\) is the graph of an
  \(L\)-Lipschitz function in the plane.
  Then for all \(\ve \ge 0\), \(0\le \bar \xi \le 1\)
\[
    \begin{split}
      \Leb(\{x \in I: \metr(x, E)\leq \ve\bar\xi\}\setminus \{x \in I:
      \metr(x, \partial I) \leq \ve\})
      &\leq \\
      & \hspace{-3cm} \bar\xi (1+L) \Leb\{x \in I: \metr(x,
      \partial I)\le \ve\}.
    \end{split}
  \]
In other words,
$\Leb((I\cap\tilde\partial_{\ve\bar\xi} E) \setminus \partial_\ve I)
\le \bar\xi (1+L) \Leb\partial_\ve I$.
\end{prop}

\begin{proof}
Suppose \(E\) is the graph of the Lipschitz function
  \(\psi:\R \to \R\).  
By a rotation of $I$ and $E$, we can suppose that
the domain of $\psi$ is the horizontal axis.

Fix \(\ve \ge 0\), \(0 \le \bar \xi \le
  1\). For $t>0$, let
  $\{A_j\}$ be a partition of the horizontal axis \(\R\) into intervals of length \(t\ve\). 
  Define
  $g_t:\R\to\R$,  $g_t|_{A_j}\equiv  (\Leb A_j)^{-1}\int_{A_j}\psi$, and
  denote \(S = \graph g_t\). 
Note that $|\psi-g_t|_\infty\le Lt\ve$.

 We claim that
\[
    \text{If } \metr(x,E)\le  \ve\bar\xi ,
     \text{ then } \metr^{\perp}(x, E)
    \le \ve\bar \xi  (1+L).
\]
  Here \(\metr^{\perp}(x, E)\)  means the vertical distance from \(x\) to
  $E$.
  Since $|\psi-g_t|_\infty\le Lt\ve$, it follows from the claim that
  \[
\{x\in I:\metr(x,E)\le\ve\bar\xi\}
\subset \{x\in I:\metr^\perp(x,S)\le\ve\xi_t\},
\]
where $\xi_t=\bar \xi  (1+L)+Lt$.

  Now applying Lemma \ref{BTGen} with $\xi = \xi_t$ on
  constant pieces \(\graph(g_t|_{A_j})\) of \(S\)
  separately and adding the contributions, we get
  \begin{align*}
    \Leb(\{x\in I:\metr(x,E)\le \ve\bar\xi\} \setminus \partial_\ve I) & \le
    \Leb(\{x\in I:\metr^\perp(x,S)\le \ve\xi_t\} \setminus \partial_\ve I) 
\\ & \le
    \xi_t \Leb\partial_{\ve}I
    =(\bar \xi  (1+L)+Lt) \Leb\partial_{\ve}I.
  \end{align*}
  Since $t>0$ is arbitrary, we obtain the desired result.
  
It remains to prove the claim.  Write $x=(x_1,x_2)$ and choose $z=(z_1,z_2)\in E$ with $|x-z|\le\ve\bar\xi$.  Let $v=(v_1,v_2)\in E$ with $v_1=x_1$.
Then 
\begin{align*}
\metr^\perp(x,E)  =|x_2-v_2| & \le |x_2-z_2|+|v_2-z_2|\le |x_2-z_2|+L|v_1-z_1|
\\ & =|x_2-z_2|+L|x_1-z_1|\le(1+L)|x-z|\le\ve\bar\xi(1+L),
\end{align*}
as required.
\end{proof}

\paragraph{Verification of controlled complexity}
Set $\Lambda_n=\sup_j |DF_{n,j}^{-1}|$ and  note that $\Lambda_n\le\uexfracone<\Lambda$.
Also, $\Lambda_n =O(n^{-(1+\alpha)})$ by~\eqref{eq:DF}.
Recall that $Y_n=\bigcup_{j=1}^{\uex^n}Y_{n,j}$ is bounded by two flat horizontal sides and two smooth vertical curves.  By Proposition~\ref{prop:xn}, the Lipschitz constants corresponding to the vertical curves are bounded by some $L_0>0$.

We choose $n_0 \ge1$ sufficiently large so that
    $\sum_{n=n_0}^{\infty} \frac{\Lambda_n}{\Lambda} (1+L_0) \le \frac18$
and then shrink $\epstail$ if needed so that if $I\subset Y$ has $\diam I\le \epstail$ then
at least one of the following holds:
\begin{itemize}
\item[(i)] $I\subset\bigcup_{n=1}^{n_0} Y_n$ and 
$I\cap \bigcup_{n=1}^{n_0}\sum_{j=1}^{\uex^n} \partial Y_{n,j}$ consists of
at most one horizontal curve $H$ and one vertical curve $V$.
\item[(ii)]  $I\subset\bigcup_{n=n_0}^\infty Y_n$.
\end{itemize}
In case~(i), $H$ is flat and $V$ is smooth.
Recall that $\Lambda<\frac13$.
Shrinking $\epstail$ further, we can suppose that $I\cap V$ is the graph of a function with Lipschitz constant $L_1$ satisfying
$3+L_1<\Lambda^{-1}$.

Let $I\subset Y$ be an open subset with $\diam I\le\epstail$.
Note that
  \[
    F_{n,j}^{-1} (\partial_{\ve}F(I\cap Y_{n,j} ))\setminus
\partial_{\ve\Lambda}I
 \subset (I\cap \partial_{\ve\Lambda_n}Y_{n,j})\setminus
\partial_{\ve\Lambda}I.
  \]
Recall that $\partial Y_{n,j}=H_{n,j}\cup V_{n,j}$ 
where $H_{n,j}$ consists of two flat horizontal edges and 
$V_{n,j}$ consists of two vertical curves.
Hence
  \begin{equation} \label{eq:partial}
    F_{n,j}^{-1} (\partial_{\ve}F(I\cap Y_{n,j} ))\setminus
\partial_{\ve\Lambda}I
 \subset \{(I\cap \tilde\partial_{\ve\Lambda_n}H_{n,j})
\setminus \partial_{\ve\Lambda}I\}
 \cup\{(I\cap \tilde\partial_{\ve\Lambda_n}V_{n,j})
\setminus \partial_{\ve\Lambda}I\}.
  \end{equation}

\noindent{\bf Case~(i)}.  Since the only intersections are with $H$ and $V$,~\eqref{eq:partial} simplifies to
\[
 \bigcup_a
F_a^{-1} (\partial_{\ve}F(I\cap a
    ))\setminus \partial_{\ve\Lambda}I
\subset 
 \{(I\cap \tilde\partial_{\ve\Lambda}H)
\setminus \partial_{\ve\Lambda}I\}
 \cup\{(I\cap \tilde\partial_{\ve\Lambda}V)
\setminus \partial_{\ve\Lambda}I\}.
\]
  By Proposition \ref{BTGenLip} (taking
  $\bar \xi = 1$ and replacing \(\ve\) by
  \(\ve\Lambda \)),
 \[
\Leb((I\cap \tilde\partial_{\ve\Lambda}V)\setminus \partial_{\ve\Lambda}I) \le 
(1+L_1) \Leb\partial_{\ve\Lambda}I.
\]
Similarly, 
$\Leb((I\cap \tilde\partial_{\ve\Lambda}H)\setminus \partial_{\ve\Lambda}I) \le 
\Leb\partial_{\ve\Lambda}I$.
Hence
\[
 \sum_a
\frac{\Leb(F_a^{-1} (\partial_{\ve}F(I\cap a
    ))\setminus \partial_{\ve\Lambda}I)}{\Leb\partial_{\ve\Lambda} I} 
\le 2+L_1<\Lambda^{-1}-1.
\]

\vspace{1ex}
\noindent{\bf Case~(ii)}.  
By~\eqref{eq:partial}.
\[
 \bigcup_a
F_a^{-1} (\partial_{\ve}F(I\cap a
    ))\setminus \partial_{\ve\Lambda}I
\subset S_H+S_V,
\]
where 
\[
S_H=\bigcup_{n=n_0}^\infty\bigcup_{j=1}^{\uex^n} \{(I\cap \tilde\partial_{\ve\Lambda_n}H_{n,j})
\setminus \partial_{\ve\Lambda}I\}, \qquad
S_V=\bigcup_{n=n_0}^\infty\bigcup_{j=1}^{\uex^n} \{(I\cap \tilde\partial_{\ve\Lambda_n}V_{n,j})
\setminus \partial_{\ve\Lambda}I\}.
\]
We estimate $S_H$ and $S_V$ separately.
The curve $V_n=\bigcup_{j=1}^{\uex^n}V_{j,n}$ is smooth with Lipschitz constant bounded by $L_0$, and
\[
S_V=\bigcup_{n=n_0}^\infty 
 \Leb \{(I\cap \tilde\partial_{\ve\Lambda_n}V_n)
\setminus \partial_{\ve\Lambda}I\}.
\]
  By Proposition \ref{BTGenLip} (taking
  \(\bar \xi = \Lambda_n/\Lambda\) and replacing \(\ve\) by
  \(\ve\Lambda \)),
\[
\Leb((I\cap \tilde\partial_{\ve\Lambda_n}V_n)\setminus
\partial_{\ve\Lambda}I) \le \frac{2\Lambda_n}{\Lambda}(1+L_0)
\Leb\partial_{\ve\Lambda}I.
\]
Hence, by the choice of $n_0$,
\[
\frac{\Leb S_V}{\Leb \partial_{\ve\Lambda}I}\le \sum_{n=n_0}^\infty \frac{2\Lambda_n}{\Lambda}(1+L_0)\le \frac14.
\]
Shrinking $\epstail$ further if necessary, it follows from the 
skew-product structure of $F$ (where vertical distances are contracted by $\uex^{-\varphi}$) that $\Lambda_n$ can be improved to $\uex^{-n}$ in the formula for $S_H$ leading to the estimate $\frac{\Leb S_H}{\Leb\partial_{\ve\Lambda}I}\le \frac14$.

Hence again we obtain the desired complexity bound
\[
 \sum_a
\frac{\Leb(F_a^{-1} (\partial_{\ve}F(I\cap a
    ))\setminus \partial_{\ve\Lambda}I)}{\Leb\partial_{\ve\Lambda} I} 
\le \frac12<\Lambda^{-1}-1.
\]

\subsection{Set \(Z\)}
\label{subsec:Tower}

The construction of $Z$ and $Z'$ proceeds as follows:
Recall that the partition elements in $\alpha^Y$ accumulate on
the left vertical side $\{\uexfrac\}\times\T$ of \(Y\).  
Let $c=1/100$ and let $S_0, S_1$ denote open squares with side lengths $c\delta$ and $2c\delta$, respectively, and centred at $l_0=(\uexfrac,0)$. Let $Z=S_0\cap Y$ and \(Z' = S_1 \cap Y\).   
It is immediate that $Z'\supset Z$, $\Leb Z<\Leb Z'$, and
\(\diam Z' \le \delta\).

Now \(Z\) is a rectangle with vertex $l_0$,
  and elements of $\alpha^Y$
accumulate at $l_0$ and shrink in diameter.  Hence there exists
\(n_0 \ge2\), $i_{0} \ge1$ such that $Y_{n,i}\subset Z$ for all \(n\ge n_0\), $i=1,\dots,\uex \, (\bmod i_n)$, where $i_n=\uex^{n-n_0}(i_0-1)$.
For \(n\ge n_0\),
\[
  \SMALL F(Z \cap \Int\{\varphi=n\}) \supset F\Big(\bigcup_{i=1}^{\uex} Y_{n,i_{n}+i}\Big)
  = Y
\]
proving~\eqref{eq:supsetY}. 
Setting
$a_1=Y_{n_0,i_{n_0}+1}$ and $a_2=Y_{n_0+1,(i_{n_0+1})+1}$,
we have $a_i\subset Z$ and $Fa_i\supset [\uexfrac,\uexfracsquare]\times\T
\supset Z'$ for $i=1,2$.
Moreover, $\varphi|_{a_1}=n_0$ and $\varphi|_{a_2}=n_0+1$, verifying~\eqref{eq:a1a2}.

\section{Boundary terms}

In this appendix we recall some standard estimates for computing
integrals around the boundary of ``rectangular'' domains.  Consider a
domain of the form
\[
  a=\{(y,\theta)\in\R\times[C,D]:\psi_1(\theta)\le y\le
  \psi_2(\theta)\},
\]
where $\psi_1,\psi_2:[C,D]\to\R$ are $C^1$ with $\psi_1<\psi_2$.
Define $M:[C,D]\to\R$,
$M(\theta)=\max\{|\psi_1'(\theta)|,|\psi_2'(\theta)|\}$.

\begin{thm} \label{thm:a} Let $v:\R^2\to\R$ be a $C^1$ function.  Then
  \[
    \begin{split}
    \int_{\partial a}|v|\le 
    & 2\Bigl\{(D-C)^{-1}+ \bigl|((1+M^2)^{1/2}+2M)/(\psi_2-\psi_1)\bigr|_\infty\Bigr\}
      |1_av|_1 \\ & \qquad +
                    \sqrt 2\bigl((1+|M|_\infty^2)^{1/2}+|M|_\infty\bigr)|1_a\nabla v|_1.
                    \end{split}
                  \]
\end{thm}

First, we consider the special case where $a$ is a rectangle.

\begin{prop} \label{prop:rectangle} Suppose that $a=[A,B]\times[C,D]$
  is a rectangle and that $v$ is $C^1$.  Write
  $\partial a=H_a\cup V_a$ where $H_a$ is the union of the two
  horizontal edges and $V_a$ is the union of the two `vertical' edges.
  Then
\[
    \int_{H_a}|v| \le  2(D-C)^{-1}|1_av|_1+ |1_a\partial v/\partial\theta|_1, 
    \quad
    \int_{V_a}|v| \le  2(B-A)^{-1}|1_av|_1+ |1_a\partial v/\partial y|_1.
  \]
  Consequently,
  $\int_{\partial a}|v|\le K|1_av|_1+\sqrt2|1_a\nabla v|_1$ where
  $K=2(B-A)^{-1}+2(D-C)^{-1}$.
\end{prop}

\begin{proof}
  We give the details for the horizontal edges.  The vertical edges
  are dealt with in the identical manner.  The final statement follows
  from the fact that $|x|+|y|\le \sqrt 2(x^2+y^2)^\frac12$.

  Note that
  $\int_{H_a}|v| =\int_A^B|v(y,C)|\,dy+\int_A^B|v(y,D)|\,dy$.
  (Throughout, we work in the coordinate system $(y,\theta)$.)  On the
  bottom edge,
  \[
    \begin{split} \int_A^B|v(y,C)| & \,dy =
      [(D-C)/2)]^{-1}\int_A^B\Bigl\{\int_C^{(C+D)/2}|v(y,C)|\,d\theta\Bigr\}\,dy
      \\
      & \le [(D-C)/2)]^{-1}\Bigl\{|1_{a_1}v|_1+
      \int_A^B\Bigl\{\int_C^{(C+D)/2}|v(y,\theta)-v(y,C)|\,d\theta\Bigr\}\,dy\Bigr\}
    \end{split}
  \]
  where $a_1$ is the rectangle $[A,B]\times[C,(C+D)/2]$.  But
  \[
    |v(y,\theta)-v(y,C)| =\Bigl|\int_C^\theta \frac{\partial
      v}{\partial \theta}(y,\psi)\,d\psi\Bigr| \le \int_C^{(C+D)/2}
    \Bigl|\frac{\partial v}{\partial \theta}(y,\psi)\Bigr|\,d\psi,
  \]
  and it follows that
  \[
    \int_A^B|v(y,C)|\,dy \le 2(D-C)^{-1}|1_{a_1}v|_1+ |1_{a_1}\partial
    v/\partial\theta|_1.
  \]
  The same estimate holds for the top edge $\int_A^B|v(y,D)|\,dy$, but
  with $a_1$ replaced by $a_2=[A,B]\times[(C+D)/2,D]$.  Since
  $a_1\cup a_2=a$ we obtain the required estimate for $\int_{H_a}|v|$.
\end{proof}

To prove Theorem~\ref{thm:a}, we introduce the diffeomorphism
$g:a\to[-1,1]\times[C,D]$ given by
\[
  g(y,\theta)=\Bigl(\frac{2y-(\psi_2(\theta)+\psi_1(\theta))}{\psi_2(\theta)-\psi_1(\theta)},\theta\Bigr),
  \quad g^{-1}(y,\theta)=(h(y,\theta),\theta),
\]
where
$h(y,\theta)=\frac12(\psi_2(\theta)-\psi_1(\theta))y+\frac12(\psi_2(\theta)+\psi_1(\theta))$.
Note that $Jg=1/\partial_yh=2/(\psi_2-\psi_1)$.

\begin{prop}
  $|\partial_\theta h(y,\theta)|\le
  M(\theta)=\max\{|\psi_1'(\theta)|,|\psi_2'(\theta)|\}$ for all
  $y\in[-1,1]$, $\theta\in[C,D]$.
\end{prop}

\begin{proof}
  Write
  $\partial_\theta
  h=\frac12(\psi_2'-\psi_1')y+\frac12(\psi_2'+\psi_1')$.  If
  $\psi_2'>\psi_1'$, then the maximum value $m_2$ and minimum value
  $m_1$ are obtained at $y=1$ and $y=-1$ respectively yielding
  $m_2=\psi_2'$ and $m_1=\psi_1'$ respectively.  The values are
  reversed if $\psi_2'<\psi_1'$.
\end{proof}

\paragraph{Vertical edges}

Let $\gamma_1$ be the left edge and ${\gamma_2}$ the right edge and
write $V_a=\gamma_1\cup \gamma_2$.

\begin{lemma} \label{lem:Va} Let $v:\R^2\to\R$ be a $C^1$ function.
  Then
  \[
    \int_{V_a}|v|\le
    2\bigl|((1+M^2)^{1/2}+M)/(\psi_2-\psi_1)\bigr|_\infty |1_av|_1+
    |(1+M^2)^{1/2}|_\infty|1_a\partial_yv|_1.
  \]
\end{lemma}

\begin{proof}
  \[
    \begin{split}
    \int_{\gamma_2} |v| & =\int_C^D |v(\psi_2(\theta),\theta)|(1+(\psi_2'(\theta))^2)^{1/2}\,d\theta
    \\ &  =\int_C^D |v(g^{-1}(1,\theta))|(1+[(\partial_\theta h)(1,\theta)]^2)^{1/2}\,d\theta
         =\int_C^D |w(1,\theta)|\,d\theta,
       \end{split}
       \]
  where $w:[-1,1]\times[C,D]\to \R$ is given by
  $w=v\circ g^{-1}\,(1+(\partial_\theta h)^2)^{1/2}$.  Similarly,
  $\int_{\gamma_1}|v|=\int_C^D|w(-1,\theta)|\,d\theta$.

  By Proposition~\ref{prop:rectangle},
  \[
    \int_{V_a} |v|\le
    |1_{[-1,1]\times[C,D]}w|_1+|1_{[-1,1]\times[C,D]}\partial_y w|_1.
  \]

  For the first term,
  \[
  \begin{split}
    |1_{[-1,1]\times[C,D]}w|_1 & =\int_{[-1,1]\times[C,D]} |v\circ g^{-1}|\,(1+(\partial_\theta h)^2)^{1/2}
    \\ &  \le \int_{[-1,1]\times[C,D]} |v\circ g^{-1}|\,(1+M^2)^{1/2} 
         = \int_a |v|\,(1+M^2)^{1/2}Jg 
    \\ & \le 2\int_a |v|\,(1+M^2)^{1/2}/(\psi_2-\psi_1)
         \le 2|(1+M^2)^{1/2}/(\psi_2-\psi_1)|_\infty |1_av|_1.
       \end{split}
       \]
       Next, we have
       \[
  \begin{split}
    \partial_yw & =(\partial_yv)\circ g^{-1}\,\partial_yh (1+(\partial_\theta h)^2)^{1/2}+v\circ g^{-1}(1+(\partial_\theta h)^2)^{-1/2}\partial_\theta h\,\partial_\theta\partial_yh.
  \end{split}
  \]
  Hence
  \[
  \begin{split}
    |1_{[-1,1]\times[C,D]}\partial_yw|_1 
    & \le I_1+I_2
  \end{split}
  \]
  where
  \[
  \begin{split}
    I_1 & =\int_{[-1,1]\times[C,D]} |(\partial_yv)\circ g^{-1}|\,|\partial_yh|\,(1+(\partial_\theta h)^2)^{1/2}  
          \\ & =\int_a |\partial_yv|\,(1+(\partial_\theta h)^2\circ g)^{1/2} 
         \le |(1+M^2)^{1/2}|_\infty|1_a\partial_y v|_1,
       \end{split}
       \]
       and
       \[
  \begin{split}
    I_2 & =\int_{[-1,1]\times[C,D]}|v\circ g^{-1}|\,(1+(\partial_\theta h)^2)^{-1/2}|\partial_\theta h|\,|\partial_\theta\partial_yh|
          \le \int_{[-1,1]\times[C,D]}|v\circ g^{-1}|\,|\partial_\theta\partial_yh|  \\ &
                                                                                          = \int_a |v|\,|\partial_\theta\partial_yh|\,Jg
                                                                                          =\int_a|v|\,|\psi_2'-\psi_1'|/(\psi_2 -\psi_1)
    \\ & \le |(\psi_2'-\psi_1')/(\psi_2-\psi_1)|_\infty |1_av|_1
         \le 2|M/(\psi_2-\psi_1)|_\infty |1_av|_1.
       \end{split}
       \]

\vspace{-5ex}
\end{proof}

\paragraph{Horizontal edges}

Next we let $H_a$ denote the union of the horizontal edges.

\begin{lemma} \label{lem:Ha} Let $v:\R^2\to\R$ be a $C^1$ function.
  Then
  \[
    \int_{H_a}|v|\le
    2\Bigl\{(D-C)^{-1}+\bigl|M/(\psi_2-\psi_1)\bigr|_\infty\Bigr\}
    |1_av|_1+|1_a\partial_\theta v|_1+|M|_\infty |1_a\partial_yv|_1.
  \]
\end{lemma}

\begin{proof}
  For the bottom edge, we write
  \[
    \int_{\psi_1(C)}^{\psi_2(C)}|v(t,C)|\,dt= \int_{-1}^1
    |v(h(y,C),C)||(\partial_yh)(y,C)|\,dy =\int_{-1}^1| w(y,C)|\,dy,
  \]
  where $w=v\circ g^{-1}\partial_yh$.  Similarly,
  $\int_{\psi_1(D)}^{\psi_2(D)}|v(t,D)|\,dt =\int_{-1}^1|
  w(y,D)|\,dy$.

  By Proposition~\ref{prop:rectangle},
  \[
    \int_{H_a}|v(t,C)|\,dt\le 2(D-C)^{-1}|1_{[-1,1]\times[C,D]}w|_1+
    |1_{[-1,1]\times[C,D]}\partial_\theta w|_1.
  \]

  Now
  \[
    |1_{[-1,1]\times[C,D]}w|_1= \int_{[-1,1]\times[C,D]}|v\circ
    g^{-1}|\,|\partial_yh| =\int_a|v|\,|\partial_yh|Jg =|1_av|_1.
  \]
  Also,
  \[
    \partial_\theta w=\partial_y v\circ g^{-1}\,\partial_\theta
    h\,\partial_yh+\partial_\theta v\circ g^{-1}\,\partial_yh+v\circ
    g^{-1}\,\partial_\theta\partial_yh,
  \]
  and so
\[
    |1_{[-1,1]\times[C,D]}\partial_\theta w|_1\le I_1+I_2+I_3,
  \]
  where
  \[
    \begin{split}
    I_1 & =
          \int_{[-1,1]\times[C,D]}|\partial_yv\circ g^{-1}||\partial_\theta h||\partial_yh|
          =\int_a |\partial_yv| |\partial_\theta h|\circ g |\partial_yh| Jg
    \\ & =\int_a |\partial_yv| |\partial_\theta h|\circ g 
         \le |M|_\infty |1_a\partial_yv|_1,
       \end{split}
       \]
  \[
    I_2= \int_{[-1,1]\times[C,D]}|\partial_\theta v\circ
    g^{-1}||\partial_yh| =\int_a|\partial_\theta v||\partial_yh|Jg
    =\int_a|\partial_\theta v|=|1_a\partial_\theta v|_1,
  \]
  and
  \[
    \begin{split}
    I_3 & 
          =\int_{[-1,1]\times[C,D]}|v\circ g^{-1}||\partial_\theta\partial_yh|
          =\int_a |v||\partial_\theta\partial_y h|Jg
          \le |(\psi_2'-\psi_1')/(\psi_2-\psi_1)|_\infty |1_av|_1
    \\ & \le 2|M/(\psi_2-\psi_1)|_\infty |1_av|_1.
  \end{split}
  \]

\vspace{-5ex}
\end{proof}

\begin{pfof}{Theorem~\ref{thm:a}}
  We combine the contributions from Lemmas~\ref{lem:Va}
  and~\ref{lem:Ha}.  The coefficient of the $|1_av|_1$ term is
  immediate.  The remaining terms yield
  \[
    |1_a\partial_\theta
    v|_1+\bigl((1+|M|_\infty^2)^{1/2}+|M|_\infty\bigr)|1_a \partial_y
    v|_1.
  \]
  The result follows since
  $|\partial_\theta v|+|\partial_y v|\le \sqrt 2|\nabla v|$.
\end{pfof}

\section{Convergence to a stable law}
\label{sec:stable}

In this appendix, we describe a general functional-analytic framework for establishing convergence to a stable law.  Our presentation follows~\cite[Theorem~6.1]{AaronsonDenker01} with a simplification due to~\cite{Gouezel10b}.

Let $F:Y\to Y$ be an ergodic measure-preserving transformation on a probability space $(Y,\mu_Y)$ with transfer operator $R:L^1(Y)\to L^1(Y)$.
Let $\cB(Y)\subset L^1(Y)$ be a Banach space containing constant functions.
In particular, $1$ is a simple eigenvalue for $R:\cB(Y)\to\cB(Y)$.
We assume that there is a spectral gap for $R:\cB(Y)\to\cB(Y)$, so $\spec R\subset \{1\}\cup B_{\kappa}(0)$ for some $\kappa<1$.

Let $\psi\in L^1(Y)$ with $\int_Y\psi\,d\mu_Y=0$, and suppose that
there are constants $\sigma_1,\sigma_2\ge0$ with $\sigma_1+\sigma_2>0$, and $\alpha\in(1,2)$, such that
\begin{equation} \label{eq:tails}
\mu_Y(\psi>x)= (\sigma_1+o(1))x^{-\alpha} \quad\text{and}\quad
\mu_Y(\psi<-x)= (\sigma_2+o(1))x^{-\alpha}
\quad\text{as $x\to\infty$.}
\end{equation}
Define
\[
\sigma= (\sigma_1+\sigma_2)\Gamma(1-\alpha)\cos{\SMALL\frac{\alpha\pi}{2}},   \qquad \beta=(\sigma_1-\sigma_2)/(\sigma_1+\sigma_2).
\]
It follows from these assumptions on $\psi$ (see~\cite[Theorem~2.6.5]{IbragimovLinnik}) that
 \[
\int_Y e^{it\psi}\,d\mu_Y=1-\sigma|t|^\alpha(1-i\beta\sgn t\tan{\SMALL\frac{\alpha\pi}{2}})+o(|t|^\alpha)
\quad\text{as $t\to0$}. 
\]

Define the twisted transfer operators $R_t:L^1(Y)\to L^1(Y)$, $t\in\R$, by
$R_tv=R(e^{it\psi}v)$.
Our final assumption is that there exists $t_0>0$, $\alpha'\in(\frac12\alpha,1]$ and $C>0$ such that
$R_t$ restricts to an operator $R_t:\cB(Y)\to\cB(Y)$ and 
$\|R_t-R\|_\cB\le C|t|^{\alpha'}$ for all $|t|<t_0$.
Let $\psi_n=\sum_{j=0}^{n-1}\psi\circ F^j$.

\begin{thm} \label{thm:stable}
Under the above assumptions,
$n^{-1/\alpha}\psi_n\to_d \sigma^{1/\alpha}G_{\alpha,\beta}$
where $G_{\alpha,\beta}$ is the $\alpha$-stable law with characteristic 
function
$\E(e^{itG_{\alpha,\beta}})= \exp\{-|t|^\alpha(1-i\beta\sgn t\tan{\SMALL\frac{\alpha\pi}{2}})\}$.
\end{thm}

\begin{proof}
The argument is by now standard.  Since we could not find the result stated in the literature, we give the details.

Since $t\mapsto R_t:\cB(Y)\to\cB(Y)$ is continuous at $t=0$,
there exists $t_1\in(0,t_0]$, $\kappa_0\in(\kappa,1)$ and $\lambda_t\in B_1(0)$,
such that $\lambda_t$ is a simple isolated eigenvalue for $R_t$
and 
$\spec R_t\subset \{\lambda_t\}\cup B_{\kappa_0}(0)$
for all $|t|<t_1$.
Moreover, $|\lambda_t-1|\ll |t|^{\alpha'}$.

Let $w_t\in\cB(Y)$ denote the family of eigenfunctions corresponding to $\lambda_t$ with \mbox{$w_0=1$}. Shrinking $t_1$ if necessary, we can ensure that $w_t>0$.
In particular, we can normalize so that $\int_Y w_t\,d\mu_Y=1$ for all $|t|<t_1$.

Let $P_t$ be the corresponding family of spectral projections with
$P_0v=\int_Y v\,d\mu_Y$.  Again $\|P_t-P_0\|_\cB\ll |t|^{\alpha'}$.
We have
\[
R_t^n=\lambda_t^n P_t+R_t^n(I-P_t).
\]
Let $\kappa_1\in(\kappa_0,1)$.
Then there exists a constant $C>0$ and functions $a_1(t)$, $a_2(t,n)$ 
such that
\[
\int_Y R_t^n1\,d\mu_Y=\lambda_t^n(1+a_1(t))\,+\,a_2(t,n)
\]
and
\[
|a_1(t)|\le C|t|^{\alpha'},\qquad  |a_2(t,n)|\le C\kappa_1^n,
\]
for all $|t|<t_1$, $n\ge1$.

Next, 
\begin{align*}
\lambda_t & =\int_Y\lambda_tw_t\,d\mu_Y
=\int_Y R_tw_t\,d\mu_Y
=\int_Y R_t1\,d\mu_Y+\int_Y (R_t-R)(w_t-w_0)\,d\mu_Y
\\
& = \int_Y e^{it\psi}\,d\mu_Y + O(t^{2\alpha'})
 =\int_Y e^{it\psi}\,d\mu_Y=1-\sigma|t|^\alpha(1-i\beta\sgn t\tan{\SMALL\frac{\alpha\pi}{2}})+o(t^\alpha).
\end{align*}

Now fix $t\in\R$. Then
\begin{align*}
\int_Y  & e^{itn^{-1/\alpha}\psi_n}\,d\mu_Y
 =\int_YR^n(e^{itn^{-1/\alpha}\psi_n})\,d\mu_Y
=\int_Y R_{tn^{-1/\alpha}}^n1\,d\mu_Y
\\ & = \lambda_{tn^{-1/\alpha}}^n(1+a_1(tn^{-1/\alpha}))\,+\,a_2(tn^{-1/\alpha},n)
=  \lambda_{tn^{-1/\alpha}}^n(1+O(n^{-\alpha'/\alpha})) \,+\,O(\kappa_1^n)
\\ & =  \big(1-\sigma|t|^\alpha n^{-1}(1-i\beta\sgn t\tan{\SMALL\frac{\alpha\pi}{2}}\big)^n(1+O(n^{-\alpha'/\alpha})) \,+\,O(\kappa_1^n) \\ & \to 
\exp\{-\sigma|t|^\alpha(1-i\beta\sgn t\tan{\SMALL\frac{\alpha\pi}{2}})\},
\end{align*}
as $n\to\infty$.
Replacing $t$ by $t\sigma^{-1/\alpha}$,
it follows from the L\'evy continuity theorem that
$\sigma^{-1/\alpha}n^{-1/\alpha}\psi_n\to_d G_{\alpha,\beta}$.
\end{proof}

\begin{rmk} \label{rmk:NDA}  
Similarly, following~\cite{AaronsonDenker01b}, if
$\mu_Y(|\psi|>x)\sim(\sigma^2+o(1))x^{-2}$ as $x\to\infty$,
then $\lambda_t=1+\sigma^2 t^2\log|t|+o(t^2\log|t|)$.
(Here, we require that $\|R_t-R\|_\cB\le C|t|$.)  The above argument then shows that
$(n\log n)^{-1/2}\psi_n\to_d N(0,\sigma^2)$.
\end{rmk}

\begin{rmk}  
We have restricted to tails of the form $\mu_Y(|\psi|>x)=\ell(x)x^{-\alpha}$ where $\lim_{x\to\infty}\ell(x)=c$ for some $c>0$, since this suffices for our examples.  The general case with $\ell$ slowly varying goes through as in~\cite{AaronsonDenker01,AaronsonDenker01b}.
\end{rmk}

\paragraph{Acknowledgements}
The research of PE and IM was supported in part by a European Advanced
Grant {\em StochExtHomog} (ERC AdG 320977).

\end{document}